\documentclass[14pt]{article}
\usepackage[pdftex]{graphicx}
\usepackage{fixltx2e}
\usepackage[gen]{eurosym}
\usepackage{wrapfig,caption}
\usepackage{cite}
\usepackage{bm,amsmath,amsthm,tabularx,setspace,textcomp,amssymb,paralist,fontenc,stmaryrd,calligra,verbatim,ulem,latexsym,mathrsfs,color,epic}
\usepackage[colorlinks = true,
            linkcolor = blue,
            urlcolor  = blue,
            citecolor = blue,
            anchorcolor = blue]{hyperref}
\usepackage{tikz-cd}
\usetikzlibrary{knots,3d,arrows,shadings}
\usetikzlibrary{calc}
\usetikzlibrary{decorations.markings}

\theoremstyle{plain}
\newtheorem{theorem}{Theorem}[section]

\newtheorem{lemma}[theorem]{Lemma}
\newtheorem{corollary}[theorem]{Corollary}
\newtheorem{conjecture}[theorem]{Conjecture}

\theoremstyle{definition}

\newtheorem{remark}[theorem]{Remark}

\numberwithin{equation}{section} % This can let your equations be numbered by section.

\begin{document}

\normalem %Because we have ulem, underline is the automatic emphasis; this fixes that.

\author{Marc Schilder}

\date{\today}

\title{The $A$-Polynomial and Knot Volume}

\maketitle

\thispagestyle{empty}
\begin{spacing}{1.25}
\maketitle

\begin{abstract}
In this paper, we conjecture a connection between the $A$-polynomial of a knot in $\mathbb{S}^{3}$ and the hyperbolic volume of its exterior $\mathcal{M}_{K}$ : the knots with zero hyperbolic volume are exactly the knots with an $A$-polynomial where every irreducible factor is the sum of two monomials in $L$ and $M$.
Herein, we show the forward implication and examine cases that suggest the converse may also be true.
Since the $A$-polynomial of hyperbolic knots are known to have at least one irreducible factor which is not the sum of two monomials in $L$ and $M$, this paper considers satellite knots which are graph knots and some with positive hyperbolic volume.
\end{abstract}

%\keywords{Knot Theory, $A$-Polynomial, 3-Manifolds, Satellite Knots, Hyperbolic Volume}

%\ccode{Mathematics Subject Classification 2010: 57M25, 57M27}

%%%%%%%%%%%%%%%%%%%%%%%%%%%
% SECTION 1: INTRODUCTION %
%%%%%%%%%%%%%%%%%%%%%%%%%%%

\section{Introduction}
One of the major problems in knot theory is distinguishing knots in $\mathbb{S}^{3}$.
There are many polynomial invariants, such as the Alexander polynomial, the colored Jones polynomials, and the HOMFLY polynomial, each utilizing properties of knot diagrams, knot exteriors, knot groups, etc.
The $A$-polynomial is an algebraic-geometric knot invariant closely related to the $SL_{2}\mathbb{C}$-character variety and the strongly detected boundary slopes of the knot.
Certain knot families have explicit formulas for their $A$-polynomials, such as $n$-twist knots~\cite{hs_2004}, iterated torus knots~\cite{nz_2017}, and $r$-twisted Whitehead doubles over torus knots~\cite{ruppe_2016}.
Other families of interest have non-explicit formulas such as double-twist knots~\cite{petersen_2015}, two-bridge knots~\cite{hs_2007}, $(-2,3,2n+1)$-pretzel knots~\cite{gm_2011}~\cite{ty_2004}, and some families of hyperbolic knots.
The $A$-polynomials of general satellite knots are less studied than those of hyperbolic knots and torus knots.

We call the {\it rational pseudo-graph knots} the family of knots whose $A$-polynomial factors so that each factor is the sum of two monomials in $L$ and $M$, $L^{q}M^{p}-\delta$ or $L^{q}-\delta M^{p}$ for relatively prime $p,q$ with $q>0$ and $\delta\in\{\pm1\}$,
\begin{align*}
\mathcal{G}_{\mathbb{Q}}:=&\left\{K\subset \mathbb{S}^{3}\middle|A_{K}\doteq\prod_{j\in J}(L^{q_{j}}M^{p_{j}}-\delta_{j}),p_{j},q_{j}\in\mathbb{Z},\delta_{j}\in\{\pm1\},(p_{j},q_{j})=1,q_{j}>0\right\},
\end{align*}
where $J$ is some finite indexing set.
The symbol $\doteq$ denotes equivalence up to normalization in $\mathbb{Z}[L,M]$, that is $f(L,M)\doteq g(L,M)$ if $f(L,M)=\sigma L^{a}M^{b}g(L,M)$ for some integers $a,b$ and $\sigma\in\{\pm1\}$, so $L^{q}M^{p}-\delta\doteq L^{q}-\delta M^{-p}$.
We also write the {\it reduced polynomial} obtained from $f(L,M)$ in $\mathbb{Z}[L,M]$ by removing repeated factors as $\mathrm{Red}[f(L,M)]$.

Contained inside this set of knots is the set of \textit{integer pseudo-graph knots} where each $q_{j}=1$:
\begin{align*}
\mathcal{G}_{\mathbb{Z}}:=&\left\{K\subset \mathbb{S}^{3}\middle|A_{K}\doteq\prod_{j\in J}(LM^{r_{j}}-\delta_{j}),r_{j}\in\mathbb{Z},\delta_{j}\in\{\pm1\}\right\}.
\end{align*}
As we will show in Corollary~\ref{graphknotsgz}, contained inside $\mathcal{G}_{\mathbb{Z}}$ is the set of {\it graph knots} $\mathcal{G}_{0}$, that is, knots whose complements are graph manifolds; these knots are combinations of $(p,q)$-cables and connected sums over the unknot $U$, which will be discussed further in Section~\ref{knotfam}.
Also of interest, the logarithmic Mahler measure of a multivariable polynomial $P(z_{1},\ldots,z_{n})$ is denoted by:
$$
\mathrm{m}(P):=\frac{1}{(2\pi)^{n}}\int_{[0,2\pi]^{n}}\ln\left|P\left(e^{i\theta_{1}},\ldots,e^{i\theta_{n}}\right)\right|\,\mathrm{d}\theta_{1}\cdots\mathrm{d}\theta_{n}.
$$
The logarithmic Mahler measures of knot polynomials appear to have connections to the geometry of the knot, so let the set of knots whose $A$-polynomial have logarithmic Mahler measure zero be denoted
\begin{align*}
\mathfrak{M}_{0}:=\left\{K\subset\mathbb{S}^{3}\middle|0=\mathrm{m}\left(A_{K}\right)=\frac{1}{(2\pi)^{2}}\int_{0}^{2\pi}\int_{0}^{2\pi}\ln\left|A_{K}(e^{i\theta},e^{i\phi})\right|\,\mathrm{d}\theta\,\mathrm{d}\phi\right\}.
\end{align*}
Simple computation of these integrals shows that $\mathcal{G}_{\mathbb{Q}}\subset\mathfrak{M}_{0}$, and hence the containments are given by $\mathcal{G}_{\mathbb{Z}}\subset\mathcal{G}_{\mathbb{Q}}\subset\mathfrak{M}_{0}$.

The main satellite operations considered in this paper are $(p,q)$-cables $[(p,q),K]$ and connected sums $K_{1}\#K_{2}$; however, we will also discuss certain winding number zero satellite operations, such as $n$-twisted Whitehead doubles and $(m,n)$-double twisted doubles.
For $(p,q)$-cables, the convention used is $q\geq2$ is the {\it winding number} of the cable and $p$ is any nonzero integer relatively prime to $q$.

Since it is unknown at this time whether $\mathcal{G}_{\mathbb{Z}}$ is a proper subset of $\mathcal{G}_{\mathbb{Q}}$, we will focus primarily on results about $\mathcal{G}_{0}$ and $\mathcal{G}_{\mathbb{Z}}$.
The first result is the computation of $A$-polynomials of connected sums and $(p,q)$-cables of knots in $\mathcal{G}_{\mathbb{Z}}$:
\begin{theorem}
\label{gzconnect}
For nontrivial knots $K_{1},K_{2}\in\mathcal{G}_{\mathbb{Z}}$ with $A_{K_{1}}\doteq\underset{i\in I}{\prod}\left(LM^{r_{i}}-\delta_{i}\right)$ and $A_{K_{2}}\doteq\underset{j\in J}{\prod}\left(LM^{s_{j}}-\delta_{j}\right)$ as above.
Then the $A$-polynomial of their connected sum $K_{1}\# K_{2}$ is given by:
\begin{align*}
A_{K_{1}\#K_{2}}\doteq\mathrm{Red}\left[\underset{(i,j)\in I\times J}{\prod}\left(LM^{r_{i}+s_{j}}-\delta_{i}\delta_{j}\right)\right]
\end{align*}
and so $K_{1}\#K_{2}\in\mathcal{G}_{\mathbb{Z}}$.
\end{theorem}
\begin{theorem}
\label{gzcable}
For a nontrivial knot $C\in\mathcal{G}_{\mathbb{Z}}$ with $A_{C}\doteq\underset{j\in J}{\prod}\left(LM^{r_{j}}-\delta_{j}\right)$ as above, the $A$-polynomial of the $(p,q)$-cable over $C$ is given by:
\begin{align*}
A_{[(p,q),C]}\doteq\mathrm{Red}\left[(L-1)F_{(p,q)}(L,M)\underset{j\in J}{\prod}\left(LM^{r_{j}q^{2}}-{\delta_{j}}^{q}\right)\right],
\end{align*}
where $F_{(p,q)}(L,M)$ is defined in Remark \ref{iteratedtorus}, and so $[(p,q),C]\in\mathcal{G}_{\mathbb{Z}}$.
\end{theorem}
\begin{corollary}
\label{graphknotsgz}
Every graph knot is an integer pseudo-graph knot, $\mathcal{G}_{0}\subset\mathcal{G}_{\mathbb{Z}}$.
\end{corollary}
This gives us the containment $\mathcal{G}_{0}\subset\mathcal{G}_{\mathbb{Z}}\subset\mathcal{G}_{\mathbb{Q}}\subset\mathfrak{M}_{0}$.
Recall the hyperbolic volume of the exterior of a knot $\mathrm{Vol}(\mathcal{M}_{K})$ is the sum of the volumes of the hyperbolic pieces in its JSJ-decomposition.
Since the graph knots are exactly the knots whose exteriors have zero hyperbolic volume, Theorems~\ref{gzconnect} and~\ref{gzcable} imply the forward direction of the following conjecture:
\begin{conjecture}
\label{apolyvol}
A knot exterior $\mathcal{M}_{K}$ has $\mathrm{Vol}(\mathcal{M}_{K})=0$ if and only if $\mathrm{m}(A_{K})=0$.
Equivalently,
$\mathcal{G}_{0}=\mathfrak{M}_{0}$.
\end{conjecture}
This conjecture comes from the above containments and a connection between hyperbolic volume and the logarithmic Mahler measure in the case of the $A$-polynomial $A_{K}(L,M)$, as discussed in \cite{gm_2018}.
Notice that $m(P\cdot Q)=m(P)+m(Q)$ and the logarithmic Mahler measure is invariant under normalization.
By the following remarks, it is known that no hyperbolic knots will be in $\mathfrak{M}_{0}$ using the fact that the $A$-polynomial of a knot is a primitive polynomial since it can be explicitly computed via resultants~\cite{dl_2006}:
\begin{remark}\cite[Corollary 2.4]{nz_2017}
\label{nohyp}
If $K$ is a hyperbolic knot, then there is a balanced-irreducible factor $f_{K}$ of $A_{K}$ which is not the sum of two monomials in $L$ and $M$.
\end{remark}
\begin{remark}
\label{mahler0prim}
\cite[Theorem 3.10]{ew_1999}, for any primitive polynomial $F(x_{1},\ldots,x_{n})\in\mathbb{Z}[x_{1}^{\pm1},\ldots,x_{n}^{\pm1}]$, $\mathrm{m}(F)=0$ if and only if $F$ is a monomial times a product of cyclotomic polynomials evaluated on monomials.
Recall a polynomial $f(L,M)\in \mathbb{Z}[L,M]$ is {\it primitive} if its content is the unit ideal $(1)$, that is, if the greatest common divisor of its coefficients is 1.
\end{remark}
\begin{remark}
\label{nohypcomp}
In Section~\ref{knotfam}, we discuss satellite knots $K=\mathrm{Sat}(P,C,f)$ for a companion knot $C$ and a pattern knot $P$ embedded in a solid torus $V$.
By \cite[Proposition 2.7]{nz_2017}, if the winding number $w$ of the embedded pattern knot $f(P)\subset V$ is nonzero, then every balanced-irreducible factor $f_{C}|A_{C}$ extends to some factor $f_{K}|A_{K}$ given by the following
$$
f_{K}(L,M)=\begin{cases}
\mathrm{Red}\left[\mathrm{Res}_{\overline{L}}\left[f_{C}(\overline{L},M^{w}),L-\overline{L}^{w}\right]\right]&:\hspace{4pt}\deg_{\overline{L}}f_{C}(\overline{L},M)\neq0,\\
f_{C}(M^{w})&:\hspace{4pt}\deg_{\overline{L}}f_{C}(\overline{L},M)=0.
\end{cases}
$$
\end{remark}

Recall that a slope on a torus $T=\partial\mathcal{M}_{K}$ is a simple closed curve $\gamma\subset\partial\mathcal{M}_{K}$ up to isotopy which does not bound a disc in $\partial\mathcal{M}_{K}$; a slope $\gamma$ can be denoted by a number $p/q\in\mathbb{Q}\cup\{\infty\}$ where $[\gamma]=[\lambda^{q}\mu^{p}]$ for the preferred framing $(\lambda,\mu)$ of $\partial\mathcal{M}_{K}$.
Note that the slope $\infty$ corresponds to the meridian $[\mu]$.
A {\it boundary slope} of a knot $K$ is a slope $\gamma$ in $\partial\mathcal{M}_{K}$ that is also the boundary of an essential surface in the knot exterior $\mathcal{M}_{K}$; a boundary slope can also be described using a number $p/q\in\mathbb{Q}\cup\{\infty\}$, similarly.
Here, a properly embedded surface $S$ in a 3-manifold is {\it essential} if $S$ is incompressible, orientable, boundary incompressible, not boundary parallel, and not a sphere.
The set of boundary slopes of the exterior of a knot $K$ is denoted $\mathcal{BS}_{K}$.
For a link $L$ of $n$-components, the set of boundary slope tuples $\mathcal{BS}_{L}$ is a collection of tuples $(m_{1},\ldots,m_{n})$ where each $m_{i}\in\mathbb{Q}\cup\{\infty,\varnothing\}$ corresponds to the slope of an essential surface along the $i$-th boundary component of $\mathcal{M}_{L}$, with $\varnothing$ denoting non-intersection with a particular component.

For a two-variable polynomial $f(L,M)=\sum_{i,j}c_{ij}L^{i}M^{j}$ the Newton polygon is the convex hull of the set of points $\left\{(i,j)\middle|c_{ij}\neq0\right\}$, denoted $\mathrm{Newt}(f)$.
The {\it strongly detected} boundary slopes of a knot are exactly the slopes of the edges of $\mathrm{Newt}(A_{K})$.
We denote the subset of strongly detected boundary slopes of a knot $K$ by $\mathcal{DS}_{K}$ to distinguish them from $\mathcal{BS}_{K}$.
Since $\mathrm{Newt}(A_{K})$ is the Minkowski sum of the Newton polygons of its factors, a factor $(LM^{r}-\delta)|A_{K}$ with $r\in\mathbb{Z}$ and $\delta\in\{\pm1\}$ contributes $r\in\mathcal{DS}_{K}$, sometimes called a {\it killing slope}.
For a knot $K\in\mathcal{G}_{\mathbb{Z}}$, the strongly detected boundary slopes $\mathcal{DS}_{K}$ can be read off by the power of $M$ in each factor, where at most two factors $LM^{r}+1$ or $LM^{r}-1$ contribute the same killing slope $r\in\mathcal{DS}_{K}$, allowing $r\in\mathbb{Z}$ up to normalization.

By Thurston's Geometrization Theorem, knots in $\mathbb{S}^{3}$ are either torus, hyperbolic, or satellite.
Every torus knot is a graph knot, and so will be in $\mathfrak{M}_{0}$ by Corollary~\ref{graphknotsgz}.
The balanced-irreducible factor $f_{K}$ from Remark~\ref{nohyp} cannot have $\mathrm{Newt}(f_{K})$ be a single edge, and so this factor $f_{K}$ will not be a cyclotomic polynomial evaluated on a Laurent monomial in $L$ and $M$; hence $\mathfrak{M}_{0}$ contains no hyperbolic knots.
Also, satellite knots $K=\mathrm{Sat}(P,C,f)$ with a hyperbolic companion knot $C$ and embedded pattern knot $f(P)$ with nonzero winding number are not contained in $\mathfrak{M}_{0}$, since the factor $f_{C}$ from Remark~\ref{nohyp} will extend to a factor $f_{K}$ of the satellite knot with $\mathrm{m}(f_{K})>0$ by Remark~\ref{nohypcomp}.
To address Conjecture~\ref{apolyvol}, it suffices to understand which satellite knots are in $\mathfrak{M}_{0}$ and if any of them have positive hyperbolic volume.

\begin{corollary}
\label{m0nohyp}
There are no hyperbolic knots in $\mathfrak{M}_{0}$.
\end{corollary}

\begin{corollary}
\label{m0nohypsat}
If the winding number $w$ of an embedded pattern knot $f(P)\subset V$ is nonzero and $C$ is a hyperbolic companion knot, then $\mathrm{Sat}(P,C,f)$ is not in $\mathfrak{M}_{0}$.
\end{corollary}

Our primary focus will be satellite knots $\mathrm{Sat}(P,C,f)$ with $f(P)\subset V$ winding number zero and companion knot $C$ a graph knot.
Additionally, we will calculate a special case of when $C$ is the figure-eight knot and the satellite operation is the $r$-twisted Whitehead double for $-11\leq r\leq11$.
Since every knot $K$ has the factor $(L-1)|A_{K}$ corresponding to the component in the representation variety $R(\mathcal{M}_{K})$ containing the abelian representations, the nontrivial factor of $A_{K}$ is denoted by $\widetilde{A}_{K}=(L-1)^{-1}A_{K}$.
By~\cite{nz_2017}, for any satellite knot $K=\mathrm{Sat}(P,C,f)$, $A_{P}|A_{K}$ and so we denote the factor of $A_{K}$ that is not contributed by the $A$-polynomial of the pattern knot $\widetilde{F}_{K}=(A_{P})^{-1}A_{K}$, and computation of $A_{\mathrm{Sat}(P,C,f)}$ reduces to computing $A_{P}$ and $\widetilde{F}_{K}$.

For a killing slope $r\in\mathcal{DS}_{C}$, we will be interested in the knot $f(P)_{r}$ obtained from $f(P)$ in the 3-sphere $V(1/r)$ after $(1/r)$-Dehn filling; the knot exterior $\mathcal{M}_{f(P)_{r}}\cong V(1/r)-\overset{\circ}{N}(f(P))$ which is explained further in Section~\ref{zerodouble}.
There is an interesting connection between the $A$-polynomials $A_{f(P)_{r}}$ for each $r\in\mathcal{DS}_{C}$ and the $A$-polynomial of the satellite knot $A_{\mathrm{Sat}(P,C,f)}$ which suggests an approach to calculating the $A$-polynomials of many satellite knots.
\begin{theorem}
\label{wnzthm}
Let $C\in\mathcal{G}_{0}$ with strongly detected boundary slopes $\mathcal{DS}_{C}$, and let $K=\mathrm{Sat}(P,C,f)$ be a satellite knot whose embedded pattern knot $f(P)\subset V$ has winding number zero in $V$.
For each integer $r\in\mathcal{DS}_{C}$, if $V(1/r)$ is the $(1/r)$-Dehn filling of $V$, then $V(1/r)-\overset{\circ}{N}\left(f(P)\right)\cong\mathcal{M}_{f(P)_{r}}$ is the exterior the knot $f(P)_{r}$.
The $A$-polynomial of $K=\mathrm{Sat}(P,C,f)$ is given in terms of the $A$-polynomials of $f(P)_{r}$ for each $r\in\mathcal{DS}_{C}$:
$$
A_{K}=\mathrm{Red}\left[(L-1)\underset{r\in\mathcal{DS}_{C}}{\prod}\widetilde{A}_{f(P)_{r}}\right].
$$
\end{theorem}

Notice that the $A$-polynomial of the pattern knot will appear $A_{P}|A_{K}$ and agrees with this result since $(LM^{0}-1)|A_{C}$ and so $0\in\mathcal{DS}_{C}$, hence $A_{f(P)_{0}}=A_{P}$ is contained in the product on the right.
Also, for a given factor $(LM^{r}-\delta)$ for some $r\in\mathcal{DS}_{C}$, the choice of $\delta\in\{\pm1\}$ does not affect the corresponding factor, $\widetilde{A}_{f(P)_{r}}$.
Furthermore, we conjecture that this equality holds for every $C\in\mathcal{G}_{\mathbb{Z}}$:
\begin{conjecture}
\label{wnzconj}
Let $C\in\mathcal{G}_{\mathbb{Z}}$ with strongly detected boundary slopes $\mathcal{DS}_{C}$, and let $K=\mathrm{Sat}(P,C,f)$ be a satellite knot whose embedded pattern knot $f(P)\subset V$ has winding number zero in $V$.
Following the notation of Theorem~\ref{wnzthm}, the $A$-polynomial of $K=\mathrm{Sat}(P,C,f)$ is given by
$$
A_{K}=\mathrm{Red}\left[(L-1)\underset{r\in\mathcal{DS}_{C}}{\prod}\widetilde{A}_{f(P)_{r}}\right].
$$
\end{conjecture}

This conjecture will be discussed in Section~\ref{proof2} after the proof of Theorem~\ref{wnzthm}; however, since the graph knots are contained in $\mathcal{G}_{\mathbb{Z}}$, Conjecture~\ref{apolyvol} would also imply the above conjecture.
The simplest nontrivial family of $A$-polynomials from Theorem~\ref{wnzthm} are the $n$-twisted Whitehead doubles of graph knots, written in terms of the $A$-polynomials of twist knots $K(n)$:
\begin{theorem}
\label{doublellin}
Let $C\in\mathcal{G}_{0}$ and let $\mathcal{DS}_{C}$ be the set of its strongly detected boundary slopes, then the $n$-twisted Whitehead double of $C$, $D_{n}(C)$ has $A$-polynomial:
\begin{align*}
A_{D_{n}(C)}=(L-1)\underset{r\in\mathcal{DS}_{C}}{\prod}\widetilde{A}_{K(n-r)}.
\end{align*}
\end{theorem}
Notice that this theorem omits the polynomial reduction.
The general construction of the $n$-twisted Whitehead double is given by Figure~\ref{fig1} in Section~\ref{zerodouble}, but this theorem can be used to immediately find many interesting families of $n$-twisted Whitehead doubles of graph knots, such as iterated torus knots and connected sums of torus knots, in terms of the $A$-polynomials of twist knots, which have known formulas by~\cite{hs_2004} and~\cite{mathews_2014}.

Theorem~\ref{doublellin} tells us that for any nontrivial knot $C$, $D_{n}(C)\not\in\mathcal{G}_{\mathbb{Z}}$, as discussed in Section~\ref{zerodouble}.
Generalizing further, we have the following result in terms of the $A$-polynomials of double twist knots $J(2m,2n)$ whose embedding is described with Figure~\ref{fig2} in Section~\ref{zerodouble}:
\begin{theorem}
\label{doubletwisteddouble}
Let $C\in\mathcal{G}_{0}$ and let $\mathcal{DS}_{C}$ be its set of strongly detected boundary slopes, then the $(m,n)$-double twisted double of $C$, $D_{m,n}(C)$ has $A$-polynomial:
\begin{align*}
A_{D_{m,n}(C)}=(L-1)\mathrm{Red}\left[\underset{r\in\mathcal{DS}_{C}}{\prod}\widetilde{A}_{J(2m,2(n-r))}\right].
\end{align*}
\end{theorem}
Here, the $n$-twisted Whitehead double of any knot $C$ is the special case $m=1$: $D_{n}(C)=D_{1,n}(C)$.
Other examples such as $(m_{1},\ldots,m_{k},n)$-twisted two-bridge doubles and $n$-twisted pretzel doubles can be constructed in terms of $A_{f(P)_{r}}$ following Theorem~\ref{wnzthm}.
\begin{remark}
Explicit formulas for the $A$-polynomials of all $(m,n)$-double twist knots are not currently known, however when $m$ is sufficiently small or $m=n$ we have formulas from Petersen~\cite{petersen_2015}.
Also, there are known symmetries of the double twist knots, such as $J(m,n)^{\ast}=J(-m,-n)$ and $J(m,n)=J(n,m)$ so we may assume that $m$ is always even (if both $m,n$ are odd, then $J(m,n)$ has two components).
\end{remark}

\begin{remark}
In Section \ref{whiteheadtwist}, we show that the $A$-polynomial of the Whitehead double over an arbitrary knot $C$ not in $\mathfrak{M}_{0}$ is much more involved.
For the figure-eight knot $C=K(-1)$, there are already difficulties in computing $A_{D_{n}(K(-1))}$ using resultant methods (or Groebner bases).
Note that the figure-eight is the simplest case in the more general problem of computing $A_{D_{n}(K(m))}$ for twist knot $K(m)$ with $m\neq0,1$.\end{remark}

In Section~\ref{apoly}, we remind the reader of the $A$-polynomial for knots in $\mathbb{S}^{3}$ and some of their properties.
In Section~\ref{knotfam}, we describe some families of knots, including torus knots, twist knots, satellite knots, graph knots, and integer pseduo-graph knots, as well as list relevant results about those knots.
In Section~\ref{proof1}, we prove Theorems~\ref{gzconnect} and~\ref{gzcable}, showing that all graph knots are in $\mathcal{G}_{\mathbb{Z}}$.
In Section~\ref{zerodouble}, we describe winding number zero satellite operations, discuss gaps in $A_{K}$, and show some results about representation varieties over winding number zero satellite knots when the companion knot is a graph knot.
In Section~\ref{rtwistedcomps}, we describe the twisted gluing relation used for explicit computations of $A_{D_{n}(C)}$, which can be used to computationally verify the results for when $C\in\mathcal{G}_{\mathbb{Z}}$ and is necessary for the calculations of $A_{D_{r}(K(-1))}$ from Section~\ref{whiteheadtwist}.
In Section~\ref{proof2}, we prove Theorem~\ref{wnzthm} with Theorems~\ref{doublellin} and~\ref{doubletwisteddouble} as special examples with known factors, and discuss Conjecture~\ref{wnzconj}.
In Section~\ref{whiteheadtwist}, we outline the resultant method for computing the $A$-polynomials of $D_{r}(K(-1))$.
In this case, a factor $Q_{K(-1),r}(L,M)$ appears in this resultant which cannot divide the $A$-polynomial because its Newton polygon has edges with slopes not in $\mathcal{BS}_{D_{r}(K(-1))}$.
Finally, in Section~\ref{conclusion}, we summarize and offer some remarks about further directions of investigation.

%%%%%%%%%%%%%%%%%%%%%%%%%%%%%%%
% SECTION 2: The A-Polynomial %
%%%%%%%%%%%%%%%%%%%%%%%%%%%%%%%
\section{The $A$-Polynomial}
\label{apoly}
The $A$-polynomial was defined by Cooper, Culler, Gillet, Long, Shalen~\cite{cooperetal_1994}, and we remind the reader here.
For a knot $K\subset\mathbb{S}^{3}$, its knot exterior is denoted $\mathcal{M}_{K}=\mathbb{S}^{3}-\overset{\circ}{N}(K)$ and its associated knot group, $\pi_{1}(\mathcal{M}_{K})$.
Within the knot group, the peripheral subgroup is denoted $\pi_{1}(\partial\mathcal{M}_{K})\cong\langle\mu_{K}\rangle\oplus\langle\lambda_{K}\rangle$ with generators $\mu_{K}$ (the meridian) and $\lambda_{K}$ (the preferred longitude) of $\partial\mathcal{M}_{K}$, and we call $(\lambda_{K},\mu_{K})$ the preferred framing of $\mathcal{M}_{K}$; here $\lambda_{K}$ is the homologically trivial longitudinal curve in $\pi_{1}(\mathcal{M}_{K})$ up to orientation.
The $SL_{2}\mathbb{C}$-representation variety of $\mathcal{M}_{K}$ is denoted $R(\mathcal{M}_{K})=\mathrm{Hom}\,(\pi_{1}(\mathcal{M}_{K}),SL_{2}\mathbb{C})$.
Taking our representations $\rho$ up to conjugacy class, we may find representations within those conjugacy classes which are upper-triangular on the peripheral subgroup and which satisfy the following, since $\mu_{K},\lambda_{K}$ commute:
\begin{align*}
\rho(\mu_{K})&=\begin{pmatrix}
M&\ast\\
0&M^{-1}\end{pmatrix}
&\rho(\lambda_{K})&=\begin{pmatrix}
L&\ast\\
0&L^{-1}\end{pmatrix}.
\end{align*}
The set of these representations is denoted by $R_{U}(\mathcal{M}_{K})$, and the projection map $\xi:R_{U}(\mathcal{M}_{K})\to\mathbb{C}^{2}$ given by $\xi(\rho)=(L,M)$ is well-defined and the Zariski closure of the image $\overline{\mathrm{im}\,\xi}$ is a complex-curve from which we can define a two-variable polynomial $A_{K}\in\mathbb{Z}[L,M]$ (unique up to sign) with:
\begin{enumerate}[(1)]
\item
$\overline{\mathrm{im}\,\xi}$ is the zero set of $A_{K}(L,M)$, that is $\overline{\mathrm{im}\,\xi}=\mathbb{V}(A_{K})$ where $\mathbb{V}(f)$ denotes the zero locus of polynomial $f$;
\item
the polynomial $A_{K}$ has no repeated factors and is in $\mathbb{Z}[L,M]$ after nonzero scaling;
\item
the polynomial $A_{K}$ can be normalized so that the coefficients are relatively prime.
\end{enumerate}
This polynomial is the {\it$A$-polynomial of $K$}, and $A_{K}$ is known to have only even powers of $M$:
$$
A_{K}(L,M)=\underset{i,j}{\sum}a_{i,2j}L^{i}M^{2j}.
$$
Here, we will only consider knots in $\mathbb{S}^{3}$, but for a more in-depth discussion, see~\cite{cooperetal_1994}.

For $(L,M)\in\mathbb{C}^{\ast}\times\mathbb{C}^{\ast}$, denote the involution $\tau(L,M)=(L^{-1},M^{-1})$ and say that a polynomial $f(L,M)$ is {\it balanced} if $f\circ\tau\doteq f$, that is,
$$
(f\circ\tau)(L,M)=\sigma L^{a}M^{b}f(L,M)
$$
for some $a,b\in\mathbb{Z}$ and $\sigma\in\{\pm1\}$.
\begin{remark}\cite{cooperetal_1994}
\label{trivfactor}
For any knot $K$, $(L-1)|A_{K}$; that is, $0\in\mathcal{DS}_{K}$.
\end{remark}
\begin{remark}\cite{cl_1997}
\label{symmetry}
For any knot $K$, $A_{K}$ is balanced.
\end{remark}
Therefore, for any irreducible factor $f|A_{K}$, either $f$ is balanced or its involution $(f\circ\tau)$ is also factor of $A_{K}$.
We note that an irreducible factor which is the sum of two monomials in $L$ and $M$ is balanced.
\begin{remark}\cite{cl_1997}
\label{mirrored}
For any knot $K$, its mirror image $K^{\ast}$ has $A$-polynomial given by $A_{K^{\ast}}(L,M)\doteq A_{K}\left(L,M^{-1}\right)$.
\end{remark}

We will also make use of the {\it $SL_{2}\mathbb{C}$-character variety} of $\mathcal{M}_{K}$, where each character $\chi_{\rho}:\pi_{1}(\mathcal{M}_{K})\to\mathbb{C}$ is given by $\chi_{\rho}(g)=\mathrm{tr}\rho(g)$, and the character variety is denoted
$$
X(\mathcal{M}_{K})=\{\chi_{\rho}|\rho\in R(\mathcal{M}_{K})\}.
$$
A construction of the $A$-polynomial based on the character variety is provided in~\cite{cooperetal_1994}, which will be summarized here.
Note that for every balanced-irreducible factor $f_{0}|\widetilde{A}_{K}$, there is a component $X_{0}$ in $X(\mathcal{M}_{K})$ which contributes this factor.
The inclusion $i:\partial\mathcal{M}_{K}\to\mathcal{M}_{K}$ induces the map $\widehat{i}_{\ast}:X(\mathcal{M}_{K})\to X(\partial\mathcal{M}_{K})$, and the algebraic map $\tau:R(\partial\mathcal{M}_{K})\to X(\partial\mathcal{M}_{K})$ given by $\tau(\rho)=\chi_{\rho}$ restricts to a degree 2 regular surjective map on the subset $\Lambda\subset R(\partial\mathcal{M}_{K})$ consisting of representations which are diagonal on the generators $\mu_{K},\lambda_{K}$.
$$
\begin{tikzcd}
&R(\mathcal{M}_{K})\arrow[r,"\tau"]\arrow[d]&X(\mathcal{M}_{K})\arrow[d,"\widehat{i}_{\ast}"]\\
\mathbb{C}^{2}\supset\mathbb{C}^{\ast}\times\mathbb{C}^{\ast}&R(\partial\mathcal{M}_{K})\arrow[r,"\tau"]\arrow[l,"\xi"]&X(\partial\mathcal{M}_{K})\\
&\Lambda\arrow[u]\arrow[ru,"\tau|_{\Lambda}" below]\arrow[lu,"\xi"]&
\end{tikzcd}
$$
The Zariski closure $\overline{\xi\left(\left(\tau|_{\Lambda}\right)^{-1}\left(\overline{\widehat{i}_{\ast}(X_{0})}\right)\right)}=D_{0}$ is a 1-dimensional variety in $\mathbb{C}^{2}$ given by $\mathbb{V}(f_{0})=D_{0}$.
The projective completion $\widetilde{X}_{0}$ and ideal points $\widetilde{x}\in\widetilde{X}_{0}$ will be used in Section~\ref{zerodouble} in the discussion of gaps of $A_{K}$.

%%%%%%%%%%%%%%%%%%%%%%%%%%%%%%%%%%%%%
% SECTION 3: Some Families of Knots %
%%%%%%%%%%%%%%%%%%%%%%%%%%%%%%%%%%%%%
\section{Some Families of Knots}
\label{knotfam}
Let $T(p,q)$ denote the $(p,q)$-torus knot which is an embedded simple closed curve on an unknotted torus $T^{2}$ in $\mathbb{S}^{3}$ in the homotopy class $[\mu^{p}\lambda^{q}]\in\pi_{1}(T^{2})$ where $\mu,\lambda$ are the standard meridian and longitude curves on the torus and $p,q$ are relatively prime integers.
Also notice that $T(p,q)=T(q,p)$ (using the complementary solid torus in $\mathbb{S}^{3}$), so we take the $(p,q)$-torus knot so that $|p|>q\geq2$ for relatively prime $p,q$ to avoid repetition.
Notice that its mirror image $T(p,q)^{\ast}=T(-p,q)$.

The family of 2-bridge knots $J(k,\ell)$ with $k$ vertical half-twists and $\ell$ horizontal half-twists are referred to as double twist knots, depicted below; for the right-handed trefoil knot $3_{1}^{+}=T(3,2)=J(2,2)$.

$$
\begin{tikzpicture}
		\begin{knot}[
%		draft mode=crossings,
		line width=1.5pt,
		line join=round,
		clip width=1,
		scale=2,
		background color=white,
		consider self intersections,
		only when rendering/.style={
			draw=white,
			double=black,
			double distance=1.5pt,
			line cap=none
		}
		]
\strand
(-1.45,.6) to ++(.3,0);
\strand
(.55,.6) to ++(.3,0);
\strand
(0,0) to +(.4,0)
arc (-90:0:.3)
to +(0,1.2)
arc (0:90:.3)
to +(-.4,0)
arc (-270:-180:.15)
.. controls +(0,-.15) and +(0,.15) .. ++(-.3,-.3)
.. controls +(0,-.15) and +(0,.15) .. ++(.3,-.3)
arc (-180:-90:.15)
to +(.25,0)
arc (90:0:.15)
to +(0,-.3)
arc (0:-90:.15)
to ++(-1.1,0)
arc (-90:-180:.15)
to ++(0,.3)
arc (180:90:.15)
to ++(.25,0)
arc (-90:0:.15)
.. controls +(0,.15) and +(0,-.15) .. ++(.3,.3)
.. controls +(0,.15) and +(0,-.15) .. ++(-.3,.3)
arc (0:90:.15)
to ++(-.4,0)
arc (90:180:.3)
to ++(0,-1.2)
arc (-180:-90:.3)
to ++(1,0);
\flipcrossings{1,3}
\end{knot}
\draw[very thick, fill=white] (0,.7) to ++(0,-.8) to ++(-1.2,0) to ++(0,.8) to ++(1.2,0) -- cycle;
\draw[very thick, fill=white] (-.2,3.3) to ++(-.8,0) to ++(0,-1.2) to ++(.8,0) to ++(0,1.2) -- cycle;
\draw[very thick, -latex] (-2.3,1.2) to ++(.1,0);
\draw[very thick, -latex] (1.1,1.2) to ++(-.1,0);
\node[left] at (-2.9,1.2) {$a$};
\node[right] at (1.7,1.2) {$b$};
\node at (-.6,.3) {$\ell$};
\node at (-.6,2.7) {$k$};
\end{tikzpicture}
\hspace{20pt}
\begin{tikzpicture}
		\begin{knot}[
%		draft mode=crossings,
		line width=1.5pt,
		line join=round,
		clip width=1,
		scale=2,
		background color=white,
		consider self intersections,
		only when rendering/.style={
			draw=white,
			double=black,
			double distance=1.5pt,
			line cap=none
		}
		]
\strand
(-1.45,.6) to ++(.3,0);
\strand
(.55,.6) to ++(.3,0);
\strand
(0,0) to +(.4,0)
arc (-90:0:.3)
to +(0,1.2)
arc (0:90:.3)
to +(-.4,0)
arc (-270:-180:.15)
.. controls +(0,-.15) and +(0,.15) .. ++(-.3,-.3)
.. controls +(0,-.15) and +(0,.15) .. ++(.3,-.3)
arc (-180:-90:.15)
to +(.25,0)
arc (90:0:.15)
to +(0,-.3)
arc (0:-90:.15)
to ++(-.25,0)
.. controls +(-.15,0) and +(.15,0) .. ++(-.55,-.3)
.. controls +(-.15,0) and +(.15,0) .. ++(-.3,.3)
to ++(-.25,0)
arc (-90:-180:.15)
to ++(0,.3)
arc (180:90:.15)
to ++(.25,0)
arc (-90:0:.15)
.. controls +(0,.15) and +(0,-.15) .. ++(.3,.3)
.. controls +(0,.15) and +(0,-.15) .. ++(-.3,.3)
arc (0:90:.15)
to ++(-.4,0)
arc (90:180:.3)
to ++(0,-1.2)
arc (-180:-90:.3)
to ++(.4,0)
.. controls +(.15,0) and +(-.15,0) .. ++(.3,.3)
.. controls +(.15,0) and +(-.15,0) .. ++(.3,-.3);
\flipcrossings{2,3}
\end{knot}
\draw[dotted] (0,-.1) to ++(0,.8) to ++(-1.2,0) to ++(0,-.8) to ++(1.2,0);
\draw[dotted] (-.2,2.1) to ++(-.8,0) to ++(0,1.2) to ++(.8,0) to ++(0,-1.2);
\draw[very thick, -latex] (-2.3,1.2) to ++(.1,0);
\draw[very thick, -latex] (1.1,1.2) to ++(-.1,0);
\node[left] at (-2.9,1.2) {$a$};
\node[right] at (1.7,1.2) {$b$};
\node at (-.6,.9) {$+2$};
\node at (0,2.7) {$+2$};
\end{tikzpicture}
$$
The figure-eight knot is another double-twist knot, instead written as $4_{1}=J(2,-2)$.
More generally, for $n\in\mathbb{Z}$, we denote the $n$-twist knot as $K(n)=J(2,2n)$.
\begin{remark}
Here, we consider only when both $k,\ell$ are even, although there is some interest in $\ell=2n+1$.
Using symmetry properties of the double twist knots, one can rewrite $J(k,\ell)=J(\ell,k)$, $J(k,\ell)^{\ast}=J(-k,-\ell)$, and $J(2,2n+1)=J(-2,2n)$.
When $k,\ell$ are both odd, $J(k,\ell)$ is a two component link, so these are not considered here.
\end{remark}

We denote the $(p,q)$-cabling over a knot $C$ by $[(p,q),C]$, whose construction is given in~\cite{nz_2017}.
If $C=T(r,s)$ is an $(r,s)$-torus knot, we may simply denote $[(p,q),T(r,s)]=[(p,q),(r,s)]$ and refer to this as an iterated torus knot.
A general iterated torus knot is similarly denoted by $[(p_{1},q_{1}),\ldots,(p_{n},q_{n})]$ which are iterated cables over a $(p_{n},q_{n})$-torus knot.
Note that each $(p_{i},q_{i})$-cable only requires $p_{i},q_{i}$ relatively prime and $q_{i}\geq2$, but the last $T(p_{n},q_{n})$ additionally requires $\left|p_{n}\right|>q_{n}\geq2$ to be a torus knot.
We also note that the $(p,q)$-cable over the unknot is the $(p,q)$-torus knot $T(p,q)$ when $|p|>1$ and the unknot for $|p|=1$.

For two knots $K_{1},K_{2}$, we denote their connected sum $K_{1}\# K_{2}$.
Beginning with the unknot, the {\it graph knots} are then the collection of all knots closed under $(p,q)$-cabling and connected sums:
$$
\mathcal{G}_{0}:=\big\langle U\big|[(p,q),-],\#\big\rangle.
$$
Equivalently, a knot $K$ is a graph knot if and only if $\mathcal{M}_{K}$ is a graph manifold, {\it i.e.} the hyperbolic volume $\mathrm{Vol}(\mathcal{M}_{K})$ is zero.
Recall that the hyperbolic volume of a knot $K$ is the sum of the volumes of the hyperbolic pieces $\mathcal{M}_{i}$ in the JSJ-decomposition, $\mathrm{Vol}(\mathcal{M}_{K})=\sum_{i}\mathrm{Vol}(\mathcal{M}_{i})$.

The $(p,q)$-cabling and connected sum operations are examples of satellite operations.
In general, a satellite knot is a knot whose exterior contains an incompressible, non-boundary parallel torus.
These knots can be constructed from a companion knot $C\subset \mathbb{S}^{3}$, a pattern knot $P$, and a homeomorphism $f:\mathbb{S}^{3}\to\mathbb{S}^{3}$ so that $f(P)$ is contained in an unknotted solid torus $V$ satisfying
\begin{enumerate}
\item
$f(P)$ is not contained in a 3-ball in $V$,
\item
$f(P)$ not isotopic to the core curve of $V$, and
\item
$f(P)$ is isotopic to $P$ when viewed in $\mathbb{S}^{3}$.
\end{enumerate}
The gluing $\phi$ is an ``untwisted'' embedding $\phi:V\to N(C)$, that is, a homeomorphism from $V$ to a regular neighborhood of $C$ that sends the meridian of $V$ to the meridian of $N(C)$, and likewise for the preferred longitudes.
We denote the satellite knot by $\mathrm{Sat}(P,C,f)=\phi(f(P))$.

The following guarantees the existence of certain factors of the $A$-polynomial of the connected sum of two knots:
\begin{remark}\cite{nz_2017}
\label{patternfactor}
For a satellite knot $K=\mathrm{Sat}(P,C,f)$ with companion knot $C$ and pattern knot $P$, $A_{P}|A_{K}$.
\end{remark}
For a satellite knot $K=\mathrm{Sat}(P,C,f)$, we denote the factor of the $A$-polynomial not contributed by the pattern knot by $\widetilde{F}_{K}=(A_{P})^{-1}A_{K}$.
Since $K_{1}\#K_{2}$ is a satellite knot where either $K_{1}$ or $K_{2}$ can be considered as the pattern knot and the other as the companion knot, we note the following corollary.
\begin{corollary}
\label{connectedfactor}
For the connected sum $K_{1}\#K_{2}$ of two knots $K_{i}$, we have $A_{K_{1}}|A_{K_{1}\#K_{2}}$ and $A_{K_{2}}|A_{K_{1}\#K_{2}}$; in particular, $\mathrm{Red}\left[(L-1)\widetilde{A}_{K_{1}}\widetilde{A}_{K_{2}}\right]\Big|A_{K_{1}\#K_{2}}$.
\end{corollary}
However, there may be other factors in $\widetilde{F}_{K_{1}\#K_{2}}$, and so the difficulty in computing $A_{K_{1}\#K_{2}}$ is computing these factors or showing none exist.

We now focus on the {\it integer pseudo-graph knots} $\mathcal{G}_{\mathbb{Z}}$, the family of knots $K$ where every irreducible factor of $A_{K}$ is the form $(LM^{r}-\delta)$ for $r\in\mathbb{Z}$ and $\delta\in\{\pm1\}$:
\begin{align*}
\mathcal{G}_{\mathbb{Z}}:=&\left\{K\subset \mathbb{S}^{3}\middle|A_{K}\doteq\prod_{j\in J}(LM^{r_{j}}-\delta_{j}),r_{j}\in\mathbb{Z},\delta_{j}\in\{\pm1\}\right\}.
\end{align*}

Remark~\ref{trivfactor} tells us that the factor $(L-1)$ with $r=0$ and $\delta=1$ will always occur in the $A$-polynomial.
The torus knots and unknot are contained in $\mathcal{G}_{\mathbb{Z}}$ by~\cite{cooperetal_1994}, the formulas of their $A$-polynomials given below; furthermore, the formula for $A_{[(p_{1},q_{1}),\ldots,(p_{n},q_{n})]}$ from~\cite{nz_2017} given below implies that every iterated torus knot is also in $\mathcal{G}_{\mathbb{Z}}$.

\begin{remark}\cite{nz_2017}
\label{iteratedtorus}
\begin{enumerate}[(1)]
\item
The $A$-polynomial of a $(p,q)$-torus knot $T(p,q)$ is
$$
A_{T(p,q)}=(L-1)F_{(p,q)}(L,M).
$$
\item
The $A$-polynomial of an iterated torus knot $[(p_{1},q_{1}),\ldots,(p_{n},q_{n})]$ is
$$
A_{[(p_{1},q_{1}),\ldots,(p_{n},q_{n})]}=(L-1)\underset{i=1}{\overset{k}{\prod}}F_{(p_{i},q_{i})}\left(L,M^{\prod_{j=1}^{i-1}q_{j}^{2}}\right)\cdot\underset{i=k+1}{\overset{n}{\prod}}G_{(p_{i},q_{i})}\left(L,M^{\prod_{j=1}^{i-1}q_{j}^{2}}\right),
$$
\end{enumerate}
where $q_{k}$ is the first even integer in the iterated cabling and the functions $F_{(p,q)}$, $G_{(p,q)}$ are as described below:
\begin{align*}
F_{(p,q)}(L,M)&=
\begin{cases}
LM^{2p}+1&:\hspace{4pt}q=2,p>0\\
L+M^{-2p}&:\hspace{4pt}q=2,p<0\\
L^{2}M^{2pq}-1&:\hspace{4pt}q>2,p>0\\
L^{2}-M^{-2pq}&:\hspace{4pt}q>2,p<0,
\end{cases}
&G_{(p,q)}(L,M)&=
\begin{cases}
LM^{pq}-1&:\hspace{4pt}p>0\\
L-M^{-pq}&:\hspace{4pt}p<0.
\end{cases}
\end{align*}
\end{remark}
We may also consider the ``non-normalized'' forms of $F_{(p,q)},G_{(p,q)}$ as
\begin{align*}
F_{(p,q)}(L,M)&\doteq
\begin{cases}
LM^{2p}+1&:\hspace{4pt}q=2\\
L^{2}M^{2pq}-1&:\hspace{4pt}q>2,
\end{cases}
&G_{(p,q)}(L,M)&\doteq
LM^{pq}-1.
\end{align*}
It is also worth noting these polynomials are a product of cyclotomic polynomials $\Phi_{n}(t)$ evaluated on the monomial in $L$ and $M$:
\begin{align*}
F_{(p,q)}(L,M)&\doteq\begin{cases}
\Phi_{2}(LM^{2p})&:\hspace{4pt}q=2\\
\Phi_{2}(LM^{pq})\Phi_{1}(LM^{pq})&:\hspace{4pt}q>2,
\end{cases}
&G_{(p,q)}(L,M)&\doteq\Phi_{1}(LM^{pq}).
\end{align*}

%%%%%%%%%%%%%%%%%%%%%%%%%%%%%%%%%%%%%%%%%%%
% SECTION 4: Proofs of G closure theorems %
%%%%%%%%%%%%%%%%%%%%%%%%%%%%%%%%%%%%%%%%%%%
\section{Proofs of Theorems~\ref{gzconnect} and~\ref{gzcable}}
\label{proof1}
To prove Theorem~\ref{gzconnect}, we recall some ideas about connected sums and utilize the notation of an amalgamated representation $\rho_{1}\ast\rho_{2}$ from Cooper, Long~\cite{cl_1997}.
For an $SL_{2}\mathbb{C}$-representation over an amalgamated product $\rho:G_{1}\ast_{H}G_{2}\to SL_{2}\mathbb{C}$, if $\rho$ restricts to representations on the subgroups $G_{i}$ as $\rho|_{G_{i}}=\rho_{i}$ such that these representations agree along the group $H$, $\rho_{1}|_{H}=\rho_{2}|_{H}$, then we may simply write $\rho=\rho_{1}\ast\rho_{2}$ when the amalgamation is understood.

For a connected sum of knots $K_{1}\#K_{2}$, it is known that the knot exterior $\mathcal{M}_{K_{1}\#K_{2}}=\mathcal{M}_{K_{1}}\cup_{A}\mathcal{M}_{K_{2}}$ over a properly embedded gluing annulus $A$ whose boundary $\partial A$ is two meridian curves in $\partial\mathcal{M}_{K_{1}}$ and $\partial\mathcal{M}_{K_{2}}$.
In either knot exterior $\mathcal{M}_{K_{i}}$, the preferred framing can be taken to be $(\lambda_{i},\mu)$ where $\mu$ is one of the components of $\partial A$ and $\lambda_{i}$ is the boundary of a properly embedded Seifert surface $F_{i}$ in $\mathcal{M}_{K_{i}}$.

We may also isotopy the surfaces so that $F_{1}\cap A=F_{2}\cap A$ are curves from one boundary component of $A$ to the other.
A minimal Seifert surface $F$ in $\mathcal{M}_{K_{1}\#K_{2}}$ can then be taken by using the band connect sum of $F_{1}$ and $F_{2}$ along their common intersection in $A$.
The homotopy class $[\partial F]$ in $\pi_{1}(\mathcal{M}_{K_{1}})\ast_{\pi_{1}(A)}\pi_{1}(\mathcal{M}_{K_{2}})$ can be represented by the preferred longitude $\lambda=\lambda_{1}\lambda_{2}$, and therefore the preferred framing of $\mathcal{M}_{K_{1}\#K_{2}}$ is $(\lambda_{1}\lambda_{2},\mu)$ since $\mathcal{M}_{K_{i}}$ can be assumed to have a common meridian $\mu$ component of $\partial A$.

If $\rho_{i}:\pi_{1}(\mathcal{M}_{K_{i}})\to SL_{2}\mathbb{C}$ are representations which agree on the common meridian $\mu$ as above, then we may conjugate so that $\rho_{i}(\mu)$ is upper-triangular, which implies that each $\rho_{i}(\lambda_{i})$ is also upper-triangular.
Since $\rho_{1}(\mu)=\rho_{2}(\mu)$, note that the eigenvalue maps $\xi:R(\mathcal{M}_{K_{i}})\to\mathbb{C}^{2}$ will have $\xi(\rho_{i})=(L_{i},M)$.
Since these representations agree along the gluing annulus, they will extend to a representation $\rho=\rho_{1}\ast\rho_{2}\in R(\mathcal{M}_{K_{1}\#K_{2}})$ such that $\rho(\lambda_{i})=\rho_{i}(\lambda_{i})$, and therefore $\rho(\lambda)=\rho_{1}(\lambda_{1})\rho_{2}(\lambda_{2})$.
Hence, the eigenvalue map $\xi:R(\mathcal{M}_{K_{1}\#K_{2}})\to\mathbb{C}^{2}$ will satisfy $\xi(\rho)=(L_{1}L_{2},M)$, as described in Cooper, Long:
\begin{lemma}\cite{cl_1997}
\label{cooperlong}
For two knots $K_{1},K_{2}$ with representations $\rho_{i}:\pi_{1}(\mathcal{M}_{K_{i}})\to SL_{2}\mathbb{C}$ the eigenvalue map $\xi(\rho_{i})=(L_{i},M)$ extends to the representation $\rho=\rho_{1}\ast\rho_{2}$ over their connected sum if and only if $\rho_{1},\rho_{2}$ agree on the meridian.
In this case, $\xi(\rho)=(L_{1}L_{2},M)$.
\end{lemma}

\subsection*{Proof of Theorem~\ref{gzconnect}}
By Lemma \ref{cooperlong}, there is a representation $\rho=\rho_{1}\ast\rho_{2}\in R(\mathcal{M}_{K_{1}\#K_{2}})$ if and only if there are representations $\rho_{i}$ which agree along the meridian, and we find that the eigenvalue map $\xi(\rho_{i})=(L_{i},M)$ extends to $\xi(\rho)=(L_{1}L_{2},M)$ and hence we have $L=L_{1}L_{2}$.
This implies that we have the following three equations in variables $L,L_{1},L_{2},M$: $A_{K_{1}}(L_{1},M)=0$, $A_{K_{2}}(L_{2},M)=0$, and $L-L_{1}L_{2}=0$.
Assuming that $K_{1},K_{2}\in\mathcal{G}_{\mathbb{Z}}$, let $A_{K_{1}}\doteq\prod_{i\in I}(LM^{r_{i}}-\delta_{i})$ and $A_{K_{2}}\doteq\prod_{j\in J}(LM^{s_{j}}-\delta_{j})$ for $r_{i},s_{j}\in\mathbb{Z}$, $\delta_{i},\delta_{j}\in\{\pm1\}$, and finite indexing sets $I,J$.
Hence, for every pair of irreducible factors $f_{i}=(LM^{r_{i}}-\delta_{i})|A_{K_{1}}$ and $g_{j}=(LM^{s_{j}}-\delta_{j})|A_{K_{2}}$, there is a corresponding polynomial factor of $A_{K_{1}\#K_{2}}$.
If the factor $f_{i}(L_{1},M)=L_{1}-1$, then we find $L_{1}=1$ which contributes $g_{j}(L,M)|A_{K_{1}\# K_{2}}$, which is already known by Corollary \ref{connectedfactor}; similarly, the factor $g_{j}(L_{2},M)=L_{2}-1$ contributes the known factor $f_{i}(L,M)|A_{K_{1}\#K_{2}}$.

Otherwise, let $f_{i}=(L_{1}M^{r_{i}}-\delta_{i})|A_{K_{1}}$ and $g_{j}=(L_{2}M^{s_{j}}-\delta_{j})|A_{K_{2}}$ be generic factors respectively, with $r_{i},s_{j}\in\mathbb{Z}$ and $\delta_{i},\delta_{j}\in\{\pm1\}$.
Solving $f_{i}(L_{1},M)=0$ and $g_{j}(L_{2},M)=0$ for $L_{i}$ gives $L_{1}=\delta_{i}M^{-r_{i}}$ and $L_{2}=\delta_{j}M^{-s_{j}}$; hence $L=L_{1}L_{2}=(\delta_{i}M^{-r_{i}})(\delta_{j}M^{-s_{j}})$ and so $LM^{r_{i}+s_{j}}-\delta_{i}\delta_{j}=0$.
Therefore, up to normalization, $(LM^{r_{i}+s_{j}}-\delta_{i}\delta_{j})|A_{K_{1}\#K_{2}}$ and so $\mathrm{Red}\left[(L-1)\prod_{i,j}(LM^{r_{i}+s_{j}}-\delta_{i}\delta_{j})\right]$ divides $A_{K_{1}\#K_{2}}$.

To make sure that isolated points $(L_{i},M)=\xi(\rho_{i})$ for $\rho_{i}\in R(\mathcal{M}_{K_{i}})$ do not contribute new factors of $A_{K_{1}\#K_{2}}$, we let $(L_{1},M)=\xi(\rho_{1})$ be an isolated point, hence $M\in\mathbb{C}^{\ast}$ must be fixed.
If $\rho_{1}$ extends to some representation $\rho=\rho_{1}\ast\rho_{2}\in R(\mathcal{M}_{K_{1}\#K_{2}})$, then there must exist a representation $\rho_{2}\in R(\mathcal{M}_{K_{2}})$ so that $\xi(\rho_{2})=(L_{2},M)$ for some $L_{2}\in\mathbb{C}^{\ast}$, however we either have $(L_{2},M)$ also an isolated point or $(L_{2},M)\in\mathbb{V}(g_{j})$ for some factor $g_{j}|A_{K_{2}}$ which uniquely determines $L_{2}=\delta_{j}M^{-s_{j}}$.
Hence if the representation $\rho_{1}$ extends to $\rho\in R(\mathcal{M}_{K_{1}\#K_{2}})$, the point $(L_{1}L_{2},M)=\xi(\rho)$ is still an isolated point.
A similar argument shows isolated points $(L_{2},M)=\xi(\rho_{2})$ will contribute only isolated points $(L_{1}L_{2},M)$.
Thus, there are no other factors, which proves the formula for computation of $A_{K_{1}\# K_{2}}$ for $K_{i}\in\mathcal{G}_{\mathbb{Z}}$.
\null\hfill$\square$\\\\

We can use the above proof to construct an unreduced, non-normalized formula for the $A$-polynomial of connected sums of integer pseudo-graph knots noting that the  $L-1$ is one of the factors in this product.
Notice that Theorem~\ref{gzconnect} can be generalized inductively to an arbitrary number of connected sums very easily:
\begin{corollary}
\label{torconnected}
Let $K_{1},\ldots,K_{n}\in\mathcal{G}_{\mathbb{Z}}$ where $A_{K_{i}}\doteq\underset{j_{i}\in J_{i}}{\prod}(LM^{r_{j_{i}}}-\delta_{j_{i}})$ with $r_{j_{i}}\in\mathbb{Z}$ and $\delta_{j_{i}}\in\{\pm1\}$ for $i=1,\ldots,n$.
Denote by $\mathbf{j}=(j_{1},\ldots,j_{n})$ where the $i$-th component $j_{i}$ corresponds to some factor $(LM^{r_{j_{i}}}-\delta_{j_{i}})$ of $A_{K_{i}}$, and let $\mathbf{J}=J_{1}\times\cdots\times J_{n}$ be the indexing set of all such $\mathbf{j}$, then
$$
A_{\#_{i=1}^{n}K_{i}}\doteq\mathrm{Red}\left[\underset{\mathbf{j}\in\mathbf{J}}{\prod}\left(LM^{\underset{i=1}{\overset{n}{\sum}}r_{j_{i}}}-\underset{i=1}{\overset{n}{\prod}}\delta_{j_{i}}\right)\right].
$$
\end{corollary}
This implies that we may take connected sums of as many knots in $\mathcal{G}_{\mathbb{Z}}$ as desired and the resulting $A$-polynomial can be found by considering combinations of factors from each component.

Also, notice that since $\widetilde{A}_{T(p,q)}=F_{(p,q)}(L,M)$ is explicitly given by Remark~\ref{iteratedtorus}, we immediately see that Corollary \ref{torconnected} is a consequence of Theorem~\ref{gzconnect}.
\begin{remark}
It is worth noting that the $A$-polynomial does not completely distinguish knots in $\mathcal{G}_{0}$.
Different torus knots can have equivalent $A$-polynomials, for example $A_{T(10,3)}=A_{T(6,5)}$.
Furthermore, by the work of Ni and Zhang, distinct cables over torus knots can have equivalent $A$-polynomials, such as $A_{[(13,15),(11,7)]}=A_{[(65,3),(275,7)]}$.
From Theorem~\ref{gzconnect} and the immediate Corollary~\ref{torconnected}, we find that there are infinitely many distinct connected sums of torus knots with equivalent $A$-polynomials.
For example, $A_{T(15,7)\# T(17,11)}=A_{T(21,5)\# T(17,11)}$.
\end{remark}

This process can be used for arbitrary connected sums of torus knots $\#_{i=1}^{n}T(p_{i},q_{i})$ noticing that any factor $(L^{2}M^{2r}-1)=(LM^{r}+1)(LM^{r}-1)$ and each component handled separately; but when we ``combine'' any factor $(LM^{r_{1}}-\delta_{1})$ with $(L^{2}M^{2r_{2}}-1)$, we get a new factor of $(L^{2}M^{2(r_{1}+r_{2})}-1)$ independent of $\delta_{1}$.

Similar combinatorial formulas will emerge as consequences of this connected sum formula, but we now move on to the proof of closure of $\mathcal{G}_{\mathbb{Z}}$ under the $(p,q)$-cabling operation.

\begin{lemma}\cite[Theorem 2.8]{nz_2017}
\label{generalcabling}
The $(p,q)$-cabling over any companion knot $C$ for $q\geq2$, $[(p,q),C]$ has $A$-polynomial
$$
A_{[(p,q),C]}=\begin{cases}
\mathrm{Red}\left[(L-1)F_{(p,q)}(L,M)\mathrm{Res}_{\overline{L}}\left[\widetilde{A}_{C}\left(\overline{L},M^{q}\right),L-\overline{L}^{q}\right]\right]&:\deg_{L}(\widetilde{A}_{C})\neq0\\
(L-1)F_{(p,q)}(L,M)\widetilde{A}_{C}(M^{q})&:\deg_{L}(\widetilde{A}_{C})=0.
\end{cases}
$$
\end{lemma}
Since knots in $\mathcal{G}_{\mathbb{Z}}$ will not have $\deg_{L}\left(\widetilde{A}_{C}\right)=0$ unless $C=U$, we prove Theorem~\ref{gzcable}:

\subsection*{Proof of Theorem \ref{gzcable}}
Let $C\in\mathcal{G}_{\mathbb{Z}}$ such that $A_{C}\doteq\prod_{j\in J}(LM^{r_{j}}-\delta_{j})$.
In the case that $\deg_{L}(\widetilde{A}_{C})=0$, this implies that $C=U$ and so we consider $[(p,q),U]$ either a torus knot or the unknot, which will be in $\mathcal{G}_{\mathbb{Z}}$ by Remark~\ref{iteratedtorus} (1).

When $\deg_{L}(\widetilde{A}_{C})\neq0$, Lemma~\ref{generalcabling} implies that $F_{(p,q)}(L,M)|A_{[(p,q),C]}$ as before, and each factor $(LM^{r_{j}}-\delta_{j})$ of $\widetilde{A}_{C}$ contributes a factor of $A_{[(p,q),C]}$ given by the resultant $\mathrm{Res}_{\overline{L}}\left[\widetilde{A}_{C}\left(\overline{L},M^{q}\right),L-\overline{L}^{q}\right]$.
By general properties of the resultant and the definition of $\widetilde{A}_{C}$, we know
\begin{align*}
\mathrm{Res}_{\overline{L}}\left[\widetilde{A}_{C}(\overline{L},M^{q}),L-\overline{L}^{q}\right]&=\mathrm{Res}_{\overline{L}}\left[\underset{j\in J}{\prod}(\overline{L}(M^{q})^{r_{j}}-\delta_{j}),L-\overline{L}^{q}\right]\\
&=\mathrm{Red}\left[\underset{j\in J}{\prod}\mathrm{Res}_{\overline{L}}\left[\overline{L}M^{r_{j}q}-\delta_{j},L-\overline{L}^{q}\right]\right].
\end{align*}
We can take this resultant directly from the Sylvester matrix:
\begin{align*}
\mathrm{Res}_{\overline{L}}\left[\overline{L}M^{r_{j}q}-\delta_{j},L-\overline{L}^{q}\right]&\doteq
\det\begin{pmatrix}
-\delta_{j}&M^{r_{j}q}&&0\\
&\ddots&\ddots&\\
0&&-\delta_{j}&M^{r_{j}q}\\
L&0&0&-1
\end{pmatrix}\\
&\doteq(-1)^{q}L\left(M^{r_{j}q}\right)^{q}-(-\delta_{j})^{q}\\
&\doteq LM^{r_{j}q^{2}}-\delta_{j}^{q}.
\end{align*}
Again, we find that the corresponding factors of $A_{[(p,q),C]}$ will kill the integer slope $r_{j}q^{2}\in\mathbb{Z}$, and therefore, every such factor $(LM^{r_{j}q^{2}}-\delta_{j}^{q})|A_{[(p,q),C]}$.

Since all of the factors of $A_{[(p,q),C]}$ up to polynomial reduction are of this form by~\cite{nz_2017}, it follows that $[(p,q),C]\in\mathcal{G}_{\mathbb{Z}}$.
\null\hfill$\square$\\\\

As with Theorem~\ref{gzconnect}, a simple argument gives a similar formula for the $A$-polynomial of an iterated cable over an integer pseudo-graph knot:
\begin{corollary}
For $C\in\mathcal{G}_{\mathbb{Z}}$ where $A_{C}\doteq\prod_{j\in J}(LM^{r_{j}}-\delta_{j})$, where $r_{j}\in\mathbb{Z}$ and $\delta_{j}\in\{-1,1\}$ and for each $i=1,\ldots,n$, we have $p_{i},q_{i}$ relatively prime with each $q_{i}\geq2$ and $\left|p_{n}\right|>q_{n}\geq2$,
$$
A_{[(p_{1},q_{1}),\ldots,(p_{n},q_{n}),C]}\doteq
\mathrm{Red}\left[A_{[(p_{1},q_{1}),\ldots,(p_{n},q_{n})]}\underset{j\in J}{\prod}\left(LM^{r_{j}\underset{i=1}{\overset{}{\prod}}q_{i}^{2}}-\delta_{j}^{\underset{i=1}{\overset{n}{\prod}}q_{i}}\right)\right].
$$
\end{corollary}
The factor of $A_{[(p_{1},q_{1}),\ldots,(p_{n},q_{n})]}$ is consistent with Remark~\ref{patternfactor} since we may think of the iterated cabling $[(p_{1},q_{1}),\ldots,(p_{n},q_{n}),C]$ as having a pattern knot $P=[(p_{1},q_{1}),\ldots,(p_{n},q_{n})]$ when $T(p_{n},q_{n})$ is a torus knot, and therefore it follows directly from $A_{P}|A_{[(p_{1},q_{1}),\ldots,(p_{n},q_{n}),C]}$.

By Theorems \ref{gzconnect} and \ref{gzcable}, we see that $\mathcal{G}_{\mathbb{Z}}$ is closed under connected sums and $(p,q)$-cabling, and thus $\mathcal{G}_{0}\subset\mathcal{G}_{\mathbb{Z}}$; furthermore, the above corollaries provide a strategy for computing the $A$-polynomials of combinations of $(p,q)$-cables and connected sums of knots in $\mathcal{G}_{\mathbb{Z}}$.

As mentioned before, since every knot $K\in\mathcal{G}_{\mathbb{Z}}$ has an $A$-polynomial where each irreducible factor can be written as the sum of two monomials in $L$ and $M$, there are no hyperbolic knots in $\mathcal{G}_{\mathbb{Z}}$; more generally, recall there are no hyperbolic knots in $\mathfrak{M}_{0}$ by Corollary~\ref{m0nohyp}.
It suffices to understand whether any satellite knots which are not graph knots are in $\mathcal{G}_{\mathbb{Z}}$.
As we will show in Section~\ref{proof2}, Theorem~\ref{wnzthm} implies that $n$-twisted Whitehead doubles of graph knots are not in $\mathcal{G}_{\mathbb{Z}}$, as well as several other families of satellite knots.

So far, the graph knots $\mathcal{G}_{0}$ are the only known examples of knots in $\mathcal{G}_{\mathbb{Z}}$ and more widely in $\mathfrak{M}_{0}$, and because all graph knots have zero hyperbolic volume, $\mathrm{Vol}(\mathcal{M}_{K})=0$, the known examples of knots in $\mathcal{G}_{\mathbb{Z}}$ support Conjecture~\ref{apolyvol}.
Since every graph knot has logarithmic Mahler measure zero by Theorems~\ref{gzconnect} and~\ref{gzcable}, the assertion of Conjecture~\ref{apolyvol} is that the $A$-polynomial of a knot $K$ has $\mathrm{m}(A_{K})=0$ implies $K$ is a graph knot.

In the next section, we will examine winding number zero satellites of graph knots, but other examples to consider are nonzero winding number satellite knots $\mathrm{Sat}(P,C,f)$ where the ``satellite space'' $V-\overset{\circ}{N}(f(P))$ has positive hyperbolic volume, for example $K=\mathrm{Sat}(U,T(3,2),f)$, where the embedding of the unknot $U$ in $V$ is given by the closure of the following solid cylinder:

$$
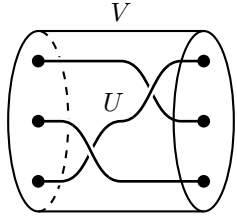
\begin{tikzpicture}[scale=2]
\draw[thick] (0,0) -- ++(1.1,0) arc (-90:90:.2 and .6) -- ++(-1.1,0) arc (90:270:.2 and .6);
\draw[thick, dashed] (0,0) arc (-90:90:.2 and .6);
		\begin{knot}[
%		draft mode=crossings,
		line width=1pt,
		line join=round,
		clip width=1,
		scale=1,
		background color=white,
		consider self intersections,
		only when rendering/.style={
			draw=white,
			double=black,
			double distance=1pt,
			line cap=none
		}
		]
\strand %strand 1
(0,.2) -- ++(.15,0)
.. controls +(.2,0) and +(-.2,0) .. ++(.4,.4)
.. controls +(.2,0) and +(-.2,0) .. ++(.4,.4)
-- ++(.15,0);
\strand %strand 2
(0,.6) -- ++(.15,0)
.. controls +(.2,0) and +(-.2,0) .. ++(.4,-.4)
-- ++(.4,0)
-- ++(.15,0);
\strand %strand 3
(0,1) -- ++(.15,0)
-- ++(.4,0)
.. controls +(.2,0) and +(-.2,0) .. ++(.4,-.4)
-- ++(.15,0);
\flipcrossings{1}
\end{knot}
\draw[thick] (1.1,0) arc (270:90:.2 and .6);
\draw[thick,fill=black] (0,.2) circle (1pt);
\draw[thick,fill=black] (0,.6) circle (1pt);
\draw[thick,fill=black] (0,1) circle (1pt);
\draw[thick,fill=black] (1.1,.2) circle (1pt);
\draw[thick,fill=black] (1.1,.6) circle (1pt);
\draw[thick,fill=black] (1.1,1) circle (1pt);
\node[above] at (.55,1.2) {$V$};
\node[above] at (.5,.6) {$U$};
\end{tikzpicture}
$$
Since $\mathrm{Vol}(\mathcal{M}_{K})=\mathrm{Vol}(V-\overset{\circ}{N}(f(U)))+\mathrm{Vol}(\mathcal{M}_{T(3,2)})>0$ (from SnapPy), we know $K$ is not a graph knot.
By Remark~\ref{nohypcomp}, since the winding number of $f(U)$ in $V$ is 3, each factor of $A_{T(3,2)}(\overline{L},M)=(\overline{L}-1)(\overline{L}M^{6}+1)$ extends to factors which are the sums of monomials in $L,M$ while $A_{P}=A_{U}=(L-1)$ contributes no nontrivial factor to $\widetilde{A}_{K}$.
The factor $(\overline{L}-1)$ contributes $\mathrm{Red}\left[\mathrm{Res}_{\overline{L}}\left[\overline{L}-1,L-\overline{L}^{3}\right]\right]=L-1$, and the factor $LM^{6}+1$ contributes $\mathrm{Red}\left[\mathrm{Res}_{\overline{L}}\left[\overline{L}(M^{3})^{6}-1,L-\overline{L}^{3}\right]\right]=LM^{54}+1$.
This implies $(LM^{54}+1)|\widetilde{F}_{\mathrm{Sat}(U,T(3,2),f)}$ for the factor $\widetilde{F}_{\mathrm{Sat}(P,C,f)}=(A_{P})^{-1}A_{\mathrm{Sat}(P,C,f)}$ mentioned in Section~\ref{knotfam}; however, this factor may contain more nontrivial factors which may not be the sum of two monomials in $L,M$.
To see whether there are other factors, we need to know whether the irreducible representations $\rho_{2}\in R^{\ast}(\mathcal{M}_{2})$ can extend to representations on the companion knot side.

%%%%%%%%%%%%%%%%%%%%%%%%%%%%%%%%%%%%%%%%%
% Section 5: r-Twisted Whitehead Double %
%%%%%%%%%%%%%%%%%%%%%%%%%%%%%%%%%%%%%%%%%
\section{Winding Number Zero Satellite Operations}
\label{zerodouble}
We call a satellite knot $K=\mathrm{Sat}(P,C,f)$ a {\it winding number zero satellite} if the embedded knot $f(P)\subset V$ has winding number zero in $V$.
An example of a winding number zero satellite is the $n$-twisted Whitehead double of $C$, $D_{n}(C)$.
To visualize the satellite operations, we illustrate the pattern knot $f(P)=\ell_{x}$ and the unknot $\ell_{y}$ so that the solid torus $V=\mathcal{M}_{\ell_{y}}$.
To construct the Whitehead double, we consider the untwisted Whitehead link $W=\ell_{x}\cup\ell_{y}$ where both $\ell_{x},\ell_{y}$ are unknots or the $n$-twisted Whitehead link:
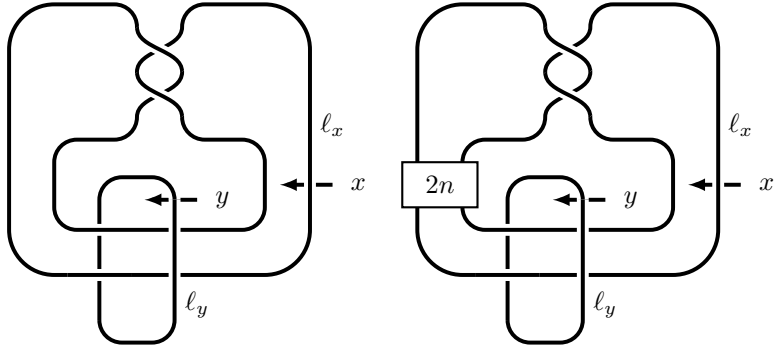
\begin{figure}[th]
$$
\begin{tikzpicture}
		\begin{knot}[
%		draft mode=crossings,
		line width=1.5pt,
		line join=round,
		clip width=1,
		scale=2,
		background color=white,
		consider self intersections,
		only when rendering/.style={
			draw=white,
			double=black,
			double distance=1.5pt,
			line cap=none
		}
		]
\strand
(.85,.6) to ++(-.3,0);
\strand
(-.35,.5) to ++(.3,0);
\strand
(-.2,.5) arc (0:90:.15)
to ++(-.2,0)
arc (90:180:.15)
to ++(0,-.8);
\strand
(0,0) to +(.4,0)
arc (-90:0:.3)
to +(0,1.2)
arc (0:90:.3)
to +(-.4,0)
arc (-270:-180:.15)
.. controls +(0,-.15) and +(0,.15) .. ++(-.3,-.3)
.. controls +(0,-.15) and +(0,.15) .. ++(.3,-.3)
arc (-180:-90:.15)
to +(.25,0)
arc (90:0:.15)
to +(0,-.3)
arc (0:-90:.15)
to ++(-1.1,0)
arc (-90:-180:.15)
to ++(0,.3)
arc (180:90:.15)
to ++(.25,0)
arc (-90:0:.15)
.. controls +(0,.15) and +(0,-.15) .. ++(.3,.3)
.. controls +(0,.15) and +(0,-.15) .. ++(-.3,.3)
arc (0:90:.15)
to ++(-.4,0)
arc (90:180:.3)
to ++(0,-1.2)
arc (-180:-90:.3)
to ++(1,0);
\strand
(-.2,.5) to ++(0,-.8)
arc (0:-90:.15)
to ++(-.2,0)
arc (-90:-180:.15);
\flipcrossings{1,2}
\end{knot}
\draw[very thick, -latex] (1.1,1.2) to ++(-.1,0);
\draw[very thick, -latex] (-.7,1) to ++(-.1,0);
\node[right] at (1.8,1.2) {$x$};
\node[right] at (0,1) {$y$};
\node at (1.7,2) {$\ell_{x}$};
\node at (-.1,-.4) {$\ell_{y}$};
\end{tikzpicture}
\hspace{10pt}
\begin{tikzpicture}
		\begin{knot}[
%		draft mode=crossings,
		line width=1.5pt,
		line join=round,
		clip width=1,
		scale=2,
		background color=white,
		consider self intersections,
		only when rendering/.style={
			draw=white,
			double=black,
			double distance=1.5pt,
			line cap=none
		}
		]
\strand
(.85,.6) to ++(-.3,0);
\strand
(-.35,.5) to ++(.3,0);
\strand
(-.2,.5) arc (0:90:.15)
to ++(-.2,0)
arc (90:180:.15)
to ++(0,-.8);
\strand
(0,0) to +(.4,0)
arc (-90:0:.3)
to +(0,1.2)
arc (0:90:.3)
to +(-.4,0)
arc (-270:-180:.15)
.. controls +(0,-.15) and +(0,.15) .. ++(-.3,-.3)
.. controls +(0,-.15) and +(0,.15) .. ++(.3,-.3)
arc (-180:-90:.15)
to +(.25,0)
arc (90:0:.15)
to +(0,-.3)
arc (0:-90:.15)
to ++(-1.1,0)
arc (-90:-180:.15)
to ++(0,.3)
arc (180:90:.15)
to ++(.25,0)
arc (-90:0:.15)
.. controls +(0,.15) and +(0,-.15) .. ++(.3,.3)
.. controls +(0,.15) and +(0,-.15) .. ++(-.3,.3)
arc (0:90:.15)
to ++(-.4,0)
arc (90:180:.3)
to ++(0,-1.2)
arc (-180:-90:.3)
to ++(1,0);
\strand
(-.2,.5) to ++(0,-.8)
arc (0:-90:.15)
to ++(-.2,0)
arc (-90:-180:.15);
\flipcrossings{1,2}
\end{knot}
\draw[very thick, -latex] (1.1,1.2) to ++(-.1,0);
\draw[very thick, -latex] (-.7,1) to ++(-.1,0);
\draw[thick,fill=white] (-2.8,.9) -- ++(1,0) -- ++(0,.6) -- ++(-1,0) -- ++(0,-.6) -- cycle;
\node[right] at (1.8,1.2) {$x$};
\node[right] at (0,1) {$y$};
\node at (1.7,2) {$\ell_{x}$};
\node at (-.1,-.4) {$\ell_{y}$};
\node at (-2.3,1.2) {$2n$};
\end{tikzpicture}
$$
\vspace*{0pt}
\caption{Untwisted Whitehead Link $W$ on the left and $n$-Twisted Whitehead Link on the right.\label{fig1}}
\end{figure}

The link exterior $\mathcal{M}_{W}=\mathbb{S}^{3}-\overset{\circ}{N}(W)$ will use the embedded $f(P)=\ell_{x}$ as the pattern knot for the untwisted double embedded into the solid torus $V=\mathcal{M}_{\ell_{y}}$.
The link group $\pi_{1}(\mathcal{M}_{W})$ has the following presentation,
$$
\pi_{1}(\mathcal{M}_{W})\cong\langle x,y|\Omega=\Omega^{\ast}\rangle
$$
where $x$ is the meridian generator coming from $\ell_{x}$, $y$ is the meridian generator coming from $\ell_{y}$, $x^{-1}=X$, $y^{-1}=Y$, the word $\Omega=yxYxyXyx$, and $\Omega^{\ast}$ denotes the reverse word of $\Omega$.
We also understand that a preferred framing of the two boundary tori $\partial N(\ell_{x})$ and $\partial N(\ell_{y})$ is given by meridians and longitudes
\begin{align*}
\mu_{x}&=x&\lambda_{x}&=XY\Omega YX=YxyXyxYX&\text{for }&\partial N(\ell_{x})\\
\mu_{y}&=y&\lambda_{y}&=YX\Omega XY=YXyxYxyX&\text{for }&\partial N(\ell_{y}).
\end{align*}
For a knot $C$, the {\it$n$-twisted Whitehead double} $D_{n}(C)$ is given by $\mathrm{Sat}(K(n),C,f)$ where the embedded twist knot $f(K(n))=\ell_{x}$ in $V=\mathcal{M}_{\ell_{y}}$ is given as shown in Figure~\ref{fig1} with $n$-vertical full-twists.

Thus, the fundamental group of the knot exterior $\mathcal{M}_{D_{n}(C)}$ is the amalgamated free-product given by the Van Kampen theorem,
$$
\pi_{1}(\mathcal{M}_{D_{n}(C)})\cong\pi_{1}(\mathcal{M}_{1})\ast_{\pi_{1}(\partial N(\ell_{y}))}\pi_{1}(\mathcal{M}_{2}),
$$
where $\mathcal{M}_{1}=\mathcal{M}_{C}$ is called the {\it companion space} and $\mathcal{M}_{2}=V-\overset{\circ}{N}(\ell_{x})$ is called the {\it satellite space}.

Generalizing slightly, the Borromean rings (shown below) give us a way of understanding a more general family of winding number zero doubles, the {\it $(m,n)$-double twisted doubles}, denoted $D_{m,n}(C)$:
\begin{figure}[th]
$$
\begin{tikzpicture}
		\begin{knot}[
%		draft mode=crossings,
		line width=1.5pt,
		line join=round,
		clip width=1,
		scale=2,
		background color=white,
		consider self intersections,
		only when rendering/.style={
			draw=white,
			double=black,
			double distance=1.5pt,
			line cap=none
		}
		]
\strand
(0,1.5) to ++(-.6,0);
\strand
(.3,1.3) to ++(-.3,0);		
\strand
(.85,.6) to ++(-.3,0);
\strand
(-.35,.5) to ++(.3,0);
\strand
(-.2,.5) arc (0:90:.15)
to ++(-.2,0)
arc (90:180:.15)
to ++(0,-.8);
\strand
(0,0) to +(.4,0)
arc (-90:0:.3)
to +(0,1.2)
arc (0:90:.3)
to +(-.4,0)
arc (-270:-180:.15)
%.. controls +(0,-.15) and +(0,.15) .. ++(-.3,-.3)
%.. controls +(0,-.15) and +(0,.15) .. ++(.3,-.3)
to +(0,-.6)
arc (-180:-90:.15)
to +(.25,0)
arc (90:0:.15)
to +(0,-.3)
arc (0:-90:.15)
to ++(-1.1,0)
arc (-90:-180:.15)
to ++(0,.3)
arc (180:90:.15)
to ++(.25,0)
arc (-90:0:.15)
%.. controls +(0,.15) and +(0,-.15) .. ++(.3,.3)
%.. controls +(0,.15) and +(0,-.15) .. ++(-.3,.3)
to +(0,.6)
arc (0:90:.15)
to ++(-.4,0)
arc (90:180:.3)
to ++(0,-1.2)
arc (-180:-90:.3)
to ++(1,0);
\strand
(-.2,.5) to ++(0,-.8)
arc (0:-90:.15)
to ++(-.2,0)
arc (-90:-180:.15);
\strand
(0,1.5)
arc (90:0:.15)
to ++(0,-.1)
arc (0:-90:.15)
to ++(-.6,0)
arc (-90:-180:.15)
to ++(0,.1)
arc (180:90:.15);
\flipcrossings{1,2,5,6}
\end{knot}
\draw[very thick, -latex] (1.1,1.2) to ++(-.1,0);
\draw[very thick, -latex] (-.7,1) to ++(-.1,0);
\draw[very thick, -latex] (0,2.6) to ++(-.1,0);
\node[right] at (1.8,1.2) {$x$};
\node[right] at (0,1) {$y$};
\node[right] at (.6,2.6) {$z$};
\node at (1.7,2) {$\ell_{x}$};
\node at (-.1,-.4) {$\ell_{y}$};
\node at (-1.7,2.6) {$\ell_{z}$};
\end{tikzpicture}
\hspace{10pt}
\begin{tikzpicture}
		\begin{knot}[
%		draft mode=crossings,
		line width=1.5pt,
		line join=round,
		clip width=1,
		scale=2,
		background color=white,
		consider self intersections,
		only when rendering/.style={
			draw=white,
			double=black,
			double distance=1.5pt,
			line cap=none
		}
		]
\strand
(.85,.6) to ++(-.3,0);
\strand
(-.35,.5) to ++(.3,0);
\strand
(-.2,.5) arc (0:90:.15)
to ++(-.2,0)
arc (90:180:.15)
to ++(0,-.8);
\strand
(0,0) to +(.4,0)
arc (-90:0:.3)
to +(0,1.2)
arc (0:90:.3)
to +(-.4,0)
arc (-270:-180:.15)
.. controls +(0,-.15) and +(0,.15) .. ++(-.3,-.3)
.. controls +(0,-.15) and +(0,.15) .. ++(.3,-.3)
arc (-180:-90:.15)
to +(.25,0)
arc (90:0:.15)
to +(0,-.3)
arc (0:-90:.15)
to ++(-1.1,0)
arc (-90:-180:.15)
to ++(0,.3)
arc (180:90:.15)
to ++(.25,0)
arc (-90:0:.15)
.. controls +(0,.15) and +(0,-.15) .. ++(.3,.3)
.. controls +(0,.15) and +(0,-.15) .. ++(-.3,.3)
arc (0:90:.15)
to ++(-.4,0)
arc (90:180:.3)
to ++(0,-1.2)
arc (-180:-90:.3)
to ++(1,0);
\strand
(-.2,.5) to ++(0,-.8)
arc (0:-90:.15)
to ++(-.2,0)
arc (-90:-180:.15);
\flipcrossings{1,2}
\end{knot}
\draw[very thick, -latex] (1.1,1.2) to ++(-.1,0);
\draw[very thick, -latex] (-.7,1) to ++(-.1,0);
\draw[thick,fill=white] (-2.8,.9) -- ++(1,0) -- ++(0,.6) -- ++(-1,0) -- ++(0,-.6) -- cycle;
\draw[very thick, fill=white] (-.2,3.3) to ++(-.8,0) to ++(0,-1.2) to ++(.8,0) to ++(0,1.2) -- cycle;
\node[right] at (1.8,1.2) {$x$};
\node[right] at (0,1) {$y$};
\node at (1.7,2) {$\ell_{x}$};
\node at (-.1,-.4) {$\ell_{y}$};
\node at (-2.3,1.2) {$2n$};
\node at (-.6,2.7) {$2m$};
\end{tikzpicture}
$$
\vspace*{0pt}
\caption{Borromean rings $B$ on the left and $(m,n)$-double twisted double satellite space on the right.\label{fig2}}
\end{figure}
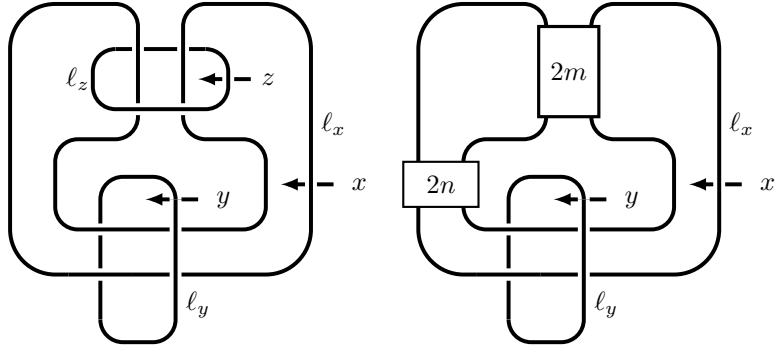

We find that the fundamental group of the Borromean rings is
$$
\pi_{1}(\mathcal{M}_{B})\cong\langle x,y,z|[x,\lambda_{x}]=[y,\lambda_{y}]=[z,\lambda_{z}]=e\rangle,
$$
where
\begin{align*}
\lambda_{x}&=ZyzY&\lambda_{y}&=zxZX&\lambda_{z}&=YXyx
\end{align*}
By performing $(1/m)$-Dehn surgery on $\mathcal{M}_{B}$ along the boundary component $\partial N(\ell_{z})$, this quotient affects the fundamental group by setting $z={\lambda_{z}}^{-m}=(YXyx)^{-m}$ and therefore the fundamental group of the satellite space on the right is given by:
$$
\pi_{1}(\mathcal{M}_{2})\cong\left\langle x,y\middle|[x,(YXyx)^{m}y(YXyx)^{-m}Y]=[y,(YXyx)^{-m}x(YXyx)^{m}X]=e\right\rangle.
$$

In full generality of winding number zero satellite knots, we will let $f(P)=\ell_{x}$ be an embedded pattern knot in $V=\mathcal{M}_{\ell_{y}}$ so that $f(P)$ has winding number zero in $V$, thus $f(P)$ bounds a Seifert surface in $V$.

To understand how killing slopes $r\in\mathcal{DS}_{C}$ extend to representations in the satellite space, we consider the quotient map obtained from $(1/r)$-Dehn filling along $\partial V$, that is $V(1/r)-\overset{\circ}{N}(f(P))=\mathcal{M}_{f(P)_{r}}$.
The quotient map $Q_{r}:\mathcal{M}_{2}\to \mathcal{M}_{f(P)_{r}}$ given by $Q_{r}(\lambda_{y}^{r}\mu_{y})=e$ induces an onto homomorphism $Q_{r\ast}:\pi_{1}(\mathcal{M}_{2})\to\pi_{1}(\mathcal{M}_{f(P)_{r}})$ also satisfying $Q_{r\ast}(\lambda_{y}^{r}\mu_{y})=e$.
Here, we denote the companion knot exterior $\mathcal{M}_{1}=\mathcal{M}_{C}$ and satellite space $\mathcal{M}_{2}=V-\overset{\circ}{N}(f(P))$.
The image of the Seifert surface $S$ under $Q_{r}$ will remain a Seifert surface of $\mathcal{M}_{f(P)_{r}}$, hence the preferred framing $(\lambda_{x},\mu_{x})$ of $\partial N(\ell_{x})$ can be thought of as the preferred framing of $\mathcal{M}_{f(P)_{r}}$.
We will refer to the boundary components of $\mathcal{M}_{2}$ as the $x$- or $y$-boundary component, denoted $\partial_{x}\mathcal{M}_{2}=\partial N(\ell_{x})$ and $\partial_{y}\mathcal{M}_{2}=\partial N(\ell_{y})$ respectively.
By the Van Kampen theorem, the fundamental group of $\mathcal{M}_{\mathrm{Sat}(P,C,f)}$ is the amalgamated free-product given by
$$
\pi_{1}(\mathcal{M}_{\mathrm{Sat}(P,C,f)})\cong\pi_{1}(\mathcal{M}_{1})\ast_{\pi_{1}(\partial_{y}\mathcal{M}_{2})}\pi_{1}(\mathcal{M}_{2}),
$$
where the gluing $\phi:\partial\mathcal{M}_{1}\to\partial_{y}\mathcal{M}_{2}$ is given by $\phi(\mu_{C})=\lambda_{y}$ and $\phi(\lambda_{C})=\mu_{y}{\lambda_{y}}^{-n}$.
Hence, we may consider $\lambda_{y}=\mu_{C}$ and $\mu_{y}=\lambda_{C}$ in $\pi_{1}(\mathcal{M}_{\mathrm{Sat}(P,C,f)})$.

Then, $\rho_{1}\in R(\mathcal{M}_{1}),\rho_{2}\in R(\mathcal{M}_{2})$ agree on the boundary by satisfying the gluing relations, $\rho_{1}(\mu_{C})=\rho_{2}(\lambda_{y})$ and $\rho_{1}(\lambda_{C})=\rho_{2}(\mu_{y})$.
Also notice that for every $\rho\in R(\mathcal{M}_{\mathrm{Sat}(P,C,f)})$, the representation will restrict to representations $\rho_{1}=\rho|_{\pi_{1}(\mathcal{M}_{1})}$ and $\rho_{2}=\rho|_{\pi_{1}(\mathcal{M}_{2})}$ which satisfy the gluing relations, hence $\rho=\rho_{1}\ast\rho_{2}$.

We note here that every representation $\sigma\in R(\mathcal{M}_{f(P)_{r}})$ will lift to a representation $\rho_{2}\in R(\mathcal{M}_{2})$ by composition with $Q_{r\ast}$:
$$
\begin{tikzcd}
\pi_{1}(\mathcal{M}_{2})\arrow[d, "Q_{r\ast}"]\arrow[dr, "\rho_{2}"]\\
\pi_{1}(\mathcal{M}_{f(P)_{r}})\arrow[r, "\sigma"]&SL_{2}\mathbb{C}
\end{tikzcd}
$$
The resulting representation $\rho_{2}$ will also satisfy $\rho_{2}(\lambda_{y}^{r}\mu_{y})=I$.
For any abelian representation $\varepsilon:\pi_{1}(\mathcal{M}_{2})\to\{\pm I\}$, $\varepsilon$ is determined by its images $\varepsilon(\mu_{x})$ and $\varepsilon(\mu_{y})$ because $[\mu_{x}]$ and $[\mu_{y}]$ generate the first homology $H_{1}(\mathcal{M}_{2};\mathbb{Z})$.
Therefore, simple calculation shows that $\rho_{2}^{\varepsilon}=\varepsilon\cdot\rho_{2}$ is still a representation, and $\rho^{\varepsilon}_{2}\in R(\mathcal{M}_{2})$ can be constructed to satisfy $\rho^{\varepsilon}_{2}(\lambda_{y}^{r}\mu_{y})=\delta I$ for $\delta\in\{\pm1\}$ by taking $\varepsilon(\mu_{x})=I$ and $\varepsilon(\mu_{y})=\delta I$.
Also notice that $[\lambda_{x}]=w[\mu_{y}]$ and $[\lambda_{y}]=w[\mu_{x}]$ in a general winding number $w$ satellite space, so for the winding number zero case, $\varepsilon(\lambda_{x})=I$ and $\varepsilon(\lambda_{y})=I$.

To prove Theorem~\ref{wnzthm}, a family of representations $\rho_{2}^{\varepsilon}\in R(\mathcal{M}_{2})$ must extend to a family of representations $\rho_{K}=\rho_{1}\ast\rho_{2}^{\varepsilon}\in R(\mathcal{M}_{K})$, for $\rho_{1}\in R(\mathcal{M}_{1})$ which agrees with $\rho_{2}^{\varepsilon}$ along the gluing torus.
Given any $\rho_{2}^{\varepsilon}(\lambda_{y})\in SL_{2}\mathbb{C}$ as above, we will show that for $C\in\mathcal{G}_{0}$ there exists a representation $\rho_{1}\in R(\mathcal{M}_{1})$ such that $\rho_{1}(\lambda_{C})=\rho_{2}^{\varepsilon}(\lambda_{y})$ and $\rho_{1}(\lambda_{C}\mu_{C}^{r})=\delta I$ for a given factor $(LM^{r}-\delta)|A_{C}$.

To address this, we say that a representation $\rho_{0}\in R(\mathcal{M}_{K})$ {\it realizes} a point $(L_{0},M_{0})\in\mathbb{V}(A_{K})$ if $\xi(\rho_{0})=(L_{0},M_{0})$.
We say that $R(\mathcal{M}_{K})$ {\it realizes} $A_{K}$ if every $(L_{0},M_{0})\in\mathbb{V}(A_{K})\cap(\mathbb{C}^{\ast})^{2}$ is realized by some $\rho_{0}\in R(\mathcal{M}_{K})$.
For a killing slope $r\in\mathcal{DS}_{C}$ with balanced-irreducible factor $f_{0}=(LM^{r}-\delta)|A_{K}$, we say that $f_{0}$ {\it has no gaps} if every $(L_{0},M_{0})\in\mathbb{V}(f_{0})\cap(\mathbb{C}^{\ast})^{2}$ is realized by a representation $\rho_{0}\in R(\mathcal{M}_{K})$ such that $\rho_{0}(\lambda_{K}\mu_{K}^{r})=\delta I$ and $\rho_{0}(\mu_{K})\neq\pm I$.
For $K\in\mathcal{G}_{0}$, we say that $A_{K}$ {\it has no gaps} if each balanced-irreducible factor $f_{0}|A_{K}$ has no gaps.

We recall the action of $\pi_{1}(\mathcal{M}_{K})$ on a simplicial tree $T$ from~\cite{cgls_1987} and the following:
\begin{remark}
\cite[Proposition 1.3.8]{cgls_1987}
\label{vertexstab}
Assume that no point of a simplicial tree $T$ is fixed by $\pi_{1}(\mathcal{M}_{K})$, then there exists an essential surface $S$ in $\mathcal{M}_{K}$ associated to the action.
Furthermore, if $C$ is a connected subcomplex of $\partial\mathcal{M}_{K}$ such that the image of $\pi_{1}(C)$ in $\pi_{1}(\partial\mathcal{M}_{K})$ is contained in a vertex stabilizer, then $S$ may be taken to be disjoint from $C$.
\end{remark}

\begin{lemma}
\label{smallnogap}
If $K$ is a small knot in $\mathbb{S}^{3}$, then $R(\mathcal{M}_{K})$ realizes $A_{K}$.
In particular, for each balanced-irreducible factor $f_{0}|A_{K}$, each point $(L_{0},M_{0})\in\mathbb{V}(f_{0})\cap(\mathbb{C}^{\ast})^{2}$ is realized by some $\rho_{0}\in R_{f_{0}}\subset R(\mathcal{M}_{K})$ in the component of $R(\mathcal{M}_{K})$ contributing $f_{0}$.
\end{lemma}

\begin{proof}
Let $K\subset\mathbb{S}^{3}$ be a small knot, that is, $\mathcal{M}_{K}$ contains no closed essential surfaces, and let $f_{0}|A_{K}$ be a balanced-irreducible factor of its $A$-polynomial with corresponding component $R_{f_{0}}\subset R(\mathcal{M}_{K})$.
By the construction of $A_{K}$ as $\overline{\mathrm{im}\,\xi}$, there are at most finitely many points $(L_{0},M_{0})\in\mathbb{V}(f_{0})\cap(\mathbb{C}^{\ast})^{2}$ which are not realized by a representation in $R_{f_{0}}$.
Assume for contradiction that $(L_{0},M_{0})$ is such a point, then there is a sequence of representations $\{\rho_{i}\}$ in $R_{f_{0}}$ such that $\xi(\rho_{i})=(L_{i},M_{i})$ with $(L_{i},M_{i})\to(L_{0},M_{0})$.
Therefore, the traces approach finite values, $\mathrm{tr}\rho_{i}(\mu_{K})=\chi_{\rho_{i}}(\mu_{K})=M_{i}+M_{i}^{-1}\to M_{0}+M_{0}^{-1}$ and $\mathrm{tr}\rho_{i}(\lambda_{K})=\chi_{\rho_{i}}(\lambda_{K})=L_{i}+L_{i}^{-1}\to L_{0}+L_{0}^{-1}$.
However, $\rho_{i}$ does not have a limit in $R_{f_{0}}$ by assumption, hence the corresponding sequence of characters $\{\chi_{\rho_{i}}\}$ do not converge in the component $X_{f_{0}}$ of the character variety.
Since $f_{0}$ is a balanced-irreducible factor, if $(L_{0},M_{0})$ is in $\mathbb{V}(f_{0})\cap\left(\mathbb{C}^{\ast}\times\mathbb{C}^{\ast}\right)$, then so is $({L_{0}}^{-1},{M_{0}}^{-1})$, and a representation $\rho_{0}\in\Lambda$ with $\xi(\rho_{0})=({L_{0}}^{-1},{M_{0}}^{-1})$ can by conjugated by $A=\begin{pmatrix}0&1\\-1&0\end{pmatrix}$ to get $\xi(A\rho_{0}A^{-1})=(L_{0},M_{0})$.

By~\cite{cs_1983}, $\{\chi_{\rho_{i}}\}$ converges in the projective completion $\widetilde{X}_{f_{0}}$ to an ideal point $\widetilde{x}_{f_{0}}$.
This implies that there is an essential surface $S\subset\mathcal{M}_{K}$ associated to this ideal point and a corresponding nontrivial action of $\pi_{1}(\mathcal{M}_{K})$ on a simplicial tree $T$; furthermore, for every $\gamma\in\pi_{1}(\partial\mathcal{M}_{K})$, the sequence $\{\chi_{\rho_{i}}(\gamma)\}$ is bounded and so $\pi_{1}(\partial\mathcal{M}_{K})$ is contained in a vertex stabilizer.
This implies by Remark~\ref{vertexstab} that the surface $S$ is disjoint from $\partial\mathcal{M}_{K}$, and since $S$ is properly embedded, $S$ must be closed.
This contradicts $K$ is a small knot, hence no such points $(L_{0},M_{0})$ can exist.
Therefore, every point $(L_{0},M_{0})\in\mathbb{V}(A_{K})\cap(\mathbb{C}^{\ast})^{2}$ has no gaps and specifically every $(L_{0},M_{0})\in\mathbb{V}(f_{0})\cap(\mathbb{C}^{\ast})^{2}$ is realized by some representation $\rho_{0}\in R_{f_{0}}$.
\end{proof}

In particular, for every torus knot $T(m,n)$, two-bridge knot, or Montesinos knot of at most three rational tangle summands, we have that each factor $f_{0}|A_{K}$ has no gaps.
By Remark~\ref{iteratedtorus}, $A_{T(m,n)}=(L-1)F_{(m,n)}(L,M)$ with
$$
F_{(m,n)}(L,M)\doteq\begin{cases}
LM^{mn}+1&:n=2\\
L^{2}M^{2mn}-1&:n>2.
\end{cases}
$$
Therefore, for $f_{0}=(LM^{r}-\delta)|A_{T(m,n)}$ if $\rho_{0}\in R_{f_{0}}$ realizes $(L_{0},M_{0})\in\mathbb{V}(f_{0})\cap(\mathbb{C}^{\ast})^{2}$ for $M_{0}\neq\pm1$, then up to conjugation, we may take $\rho_{0}(\mu_{K})$ to be a diagonal matrix; thus, $\rho_{0}(\lambda_{K}\mu_{K}^{r})=\delta I$.
The following lemma addresses when $M_{0}=\pm1$.

\begin{lemma}
\label{irrednogap}
Let $f_{0}(L,M)=(LM^{r}-\delta)|A_{K}(L,M)$ be a balanced-irreducible factor such that every $(L_{0},M_{0})\in\mathbb{V}(f_{0})\cap(\mathbb{C}^{\ast})^{2}$ is realized by a representation $\rho_{0}\in R_{f_{0}}$.
Then for $M_{0}=\pm1$, a representation $\rho_{0}\in R_{f_{0}}$ such that $\xi(\rho_{0})=(L_{0},M_{0})$ where $L_{0}=\delta M_{0}^{-r}$ for $r\neq0$ is always an irreducible representation.
In particular, $\rho_{0}(\mu_{K})\neq\pm I$.
\end{lemma}

\begin{proof}
Assume for contradiction that such a representation $\rho_{0}\in R_{f_{0}}$ is reducible, then since $R_{f_{0}}$ is at least 4-dimensional, there is a reducible nonabelian representation $\rho_{1}\in R_{f_{0}}$ such that $\chi_{\rho_{1}}=\chi_{\rho_{0}}$.
In particular, $\mathrm{tr}\rho_{1}(\mu_{K})=2M_{0}$ which is either $2$ or $-2$.
By~\cite{burde_1967} and~\cite{derham_1967}, this implies that $1$ must be a root of the Alexander polynomial of $K$, $\Delta_{K}(1)=0$; however, $\Delta_{K}(1)=\pm1$ for every knot $K$, a contradiction.
Hence, $\rho_{0}$ must be irreducible, and thus $\rho_{0}(\mu_{K})\neq\pm I$.
\end{proof}

Since every torus knot $T(m,n)$ is a small knot, we have the following corollary, which serves as the base case for our induction on the graph knots:

\begin{corollary}
\label{torusnogaps}
For every torus knot $T(m,n)$, $A_{T(m,n)}$ has no gaps.
\end{corollary}

The above corollary guarantees for each factor $f_{0}\doteq(LM^{r}-\delta)|A_{T(m,n)}$, each $(L_{0},M_{0})\in\mathbb{V}(f_{0})\cap(\mathbb{C}^{\ast})^{2}$ is realized by a representation $\rho_{0}\in R_{f_{0}}$ such that $\xi(\rho_{0})=(L_{0},M_{0})$, $\rho_{0}(\lambda\mu^{r})=\delta I$, and $\rho_{0}(\mu)\neq\pm I$.

\begin{lemma}
\label{sumnogaps}
If $K_{1},K_{2}\in\mathcal{G}_{0}$ are knots where $A_{K_{1}},A_{K_{2}}$ have no gaps, then $A_{K_{1}\#K_{2}}$ has no gaps.
\end{lemma}

\begin{proof}
Let $K=K_{1}\#K_{2}$, then each pair of factors $(L_{1}M^{r_{i}}-\delta_{i})|A_{K_{1}}$ and $(L_{2}M^{s_{j}}-\delta_{j})|A_{K_{2}}$ contributes the factor $(LM^{r_{i}+s_{j}}-\delta_{i}\delta_{j})|A_{K}$ by Theorem~\ref{gzconnect}.
For a given $M_{0}\in\mathbb{C}^{\ast}$, this determines $L_{1}=\delta_{i}M_{0}^{-r_{i}}$, $L_{2}=\delta_{j}M_{0}^{-s_{j}}$, and therefore $L_{0}=\delta_{i}\delta_{j}M^{-r_{i}-s_{j}}$.
Since each $A_{K_{i}}$ has no gaps and $L_{1},L_{2}\in\mathbb{C}^{\ast}$, there exist representations $\rho_{i}\in R(\mathcal{M}_{K_{i}})$ that realize $(L_{i},M_{0})$ with $\rho_{i}(\mu_{K})\neq\pm I$ for $i=1,2$.
Furthermore, $\rho_{1}(\lambda_{1}\mu_{K}^{r_{i}})=\delta_{i}I$ and $\rho_{2}(\lambda_{2}\mu_{K}^{s_{j}})=\delta_{j}I$, and so these representations will also satisfy $\rho_{1}(\lambda_{1})=\delta_{i}\rho_{1}(\mu_{K})^{-r_{i}}$ and $\rho_{2}(\lambda_{2})=\delta_{j}\rho_{2}(\mu_{K})^{-s_{j}}$.
If $M_{0}\neq\pm1$, that is $\mathrm{tr}\rho_{i}(\mu_{K})\neq\pm2$, then both $\rho_{i}(\mu_{K})$ are diagonalizable, so up to conjugation we have $\rho_{i}(\mu_{K})\begin{pmatrix}M_{0}&0\\0&M_{0}^{-1}\end{pmatrix}\neq\pm I$ for $i=1,2$.
If $M_{0}=\pm1$, then up to conjugation, $\rho_{i}(\mu_{K})=M_{0}\begin{pmatrix}1&1\\0&1\end{pmatrix}$ for $i=1,2$ since $(L_{i},M_{0})$ is not a gap of $A_{K_{i}}$ and so $\rho_{i}(\mu_{K})\neq\pm I$.
Hence, in either case these representations agree on the gluing annulus $A$ of the connected sum, and so $\rho_{1}\ast\rho_{2}\in R(\mathcal{M}_{K_{1}\#K_{2}})$ with $\xi(\rho_{1}\ast\rho_{2})=(L_{1}L_{2},M_{0})=(\delta_{i}\delta_{j}M_{0}^{-r_{i}-s_{j}},M_{0})$, $(\rho_{1}\ast\rho_{2})(\mu_{K})\neq\pm I$, and $(\rho_{1}\ast\rho_{2})(\lambda_{1}\lambda_{2}\mu_{K}^{r_{i}+s_{j}})=\delta_{1}\delta_{2}I$.
Therefore, $A_{K_{1}\#K_{2}}$ has no gaps.
\end{proof}

\begin{lemma}
\label{cablenogaps}
For a knot $C\in\mathcal{G}_{0}$ where $A_{C}$ has no gaps, then $A_{[(p,q),C]}$ has no gaps.
In particular, the factor $F_{(p,q)}$ has no gaps.
\end{lemma}

\begin{proof}
Letting $K=[(p,q),C]$, each factor $f_{0}(\overline{L},\overline{M})=(\overline{L}\,\overline{M}^{r}-\delta)$ of $A_{C}$ contributes a factor $g_{0}(L,M)=(LM^{rq^{2}}-\delta^{q})$ of $A_{K}$; additionally, there is the factor $F_{(p,q)}(L,M)$ of $A_{K}$.
It suffices to show that every $(L_{0},M_{0})\in\mathbb{V}(g_{0})\cap(\mathbb{C}^{\ast})^{2}$ is realized by some $\rho\in R(\mathcal{M}_{K})$ and every $(L,M)\in\mathbb{V}(F_{(p,q)})$ is realized by some $\rho\in R(\mathcal{M}_{K})$.

We begin with the factor $F_{(p,q)}$ using a modified argument of Claim 2.9 in~\cite{nz_2017}.
If $|p|>1$, then notice that by Corollary~\ref{torusnogaps}, for any $M_{0}\in\mathbb{C}^{\ast}$, there is a representation $\sigma_{M_{0}}\in R(\mathcal{M}_{T(p,q)})$ realizing $(L_{0},M_{0})\in\mathbb{V}(F_{(p,q)})\cap(\mathbb{C}^{\ast})^{2}$ with $\sigma_{M_{0}}(\lambda_{K}\mu_{K}^{pq})=\pm I$ and $\sigma_{M_{0}}(\mu_{K})\neq\pm I$.
Composing with the induced quotient homomorphism $Q_{0\ast}:\pi_{1}(\mathcal{M}_{2})\to\pi_{1}(\mathcal{M}_{T(p,q)})$ gives the representation $\rho_{2}=\sigma_{M_{0}}\circ Q_{0\ast}\in R(\mathcal{M}_{2})$ and so extend to the representation $\rho_{K}=\mathrm{id}\ast\rho_{2}\in R(\mathcal{M}_{K})$.
Hence, every $(L_{0},M_{0})\in\mathbb{V}(F_{(p,q)})\cap(\mathbb{C}^{\ast})^{2}$ is realized by some $\rho_{K}\in R(\mathcal{M}_{K})$ for $|p|>1$ with $\rho_{K}(\lambda_{K}\mu_{K}^{pq})=\pm I$ and $\rho_{K}(\mu_{K})\neq\pm I$.
If $|p|=1$, then since the quotient map would give us the unknot $T(p,q)\cong U$, we recall from the discussion in~\cite{nz_2017}, the fundamental group of the cable space for $p=\pm1$:
$$
\pi_{1}(\mathcal{M}_{2})\cong\langle\alpha,\beta|\gamma_{C}=\alpha^{q},\gamma_{C}\beta=\beta\gamma_{C}\rangle,
$$
for a Seifert fiber of $\mathcal{M}_{2}$ $\gamma_{C}=\alpha^{q}$ lying in $\partial V$, and a Seifert fiber of $\mathcal{M}_{2}$ $\gamma_{K}=\lambda_{K}\mu_{K}^{pq}$ lying in $\partial\mathcal{M}_{K}$.
As described in~\cite{nz_2017} that $\mu_{K}=\alpha\beta$ and $\lambda_{K}=\gamma_{K}\mu_{K}^{-pq}$ with $\rho(\gamma_{C})=\rho(\gamma_{K})=\pm I$ for any irreducible representation $\rho\in R(\mathcal{M}_{2})$.

Letting $(L_{0},M_{0})\in\mathbb{V}(LM^{pq}-1)$ for $q>2$, then because $\mathcal{M}_{C}(p/q)$ is a homology sphere, by \cite{km_2004}, there must exist an irreducible $SU(2)$-representation of $\pi_{1}(\mathcal{M}_{C})$ and hence an irreducible representation $\rho_{1}\in R(\mathcal{M}_{C}(p/q))\subset R(\mathcal{M}_{C})$ satisfying $\rho_{1}(\lambda_{C}^{q}\mu_{C}^{p})=I$ and $\mathrm{tr}\rho_{1}(\lambda_{C})\neq\pm2$.
Hence, up to conjugation we may assume that $\rho_{1}$ satisfies the following:
\begin{align*}
\rho_{1}(\lambda_{C})&=\begin{pmatrix}\ell&0\\0&\ell^{-1}\end{pmatrix}&\rho_{1}(\mu_{C})&=\begin{pmatrix}\ell^{-pq}&0\\0&\ell^{pq}\end{pmatrix},
\end{align*}
for some choice of $\ell\neq\pm1$ and $\ell^{q}\neq\pm1$.
We define $\rho_{2}(\alpha)=ABA^{-1}$ for matrices $A,B\in SL_{2}\mathbb{C}$ such that $A=\begin{pmatrix}a&b\\c&d\end{pmatrix}$, $B=\begin{pmatrix}z&0\\0&z^{-1}\end{pmatrix}$ with $z^{q}=1$ and $z\neq\pm1$.
Simple calculation shows we may take $a\in\mathbb{C}^{\ast}$, $b=1$, $c=\frac{M_{0}+M_{0}^{-1}-\ell z-\ell^{-1}z^{-1}}{(\ell-\ell^{-1})(z-z^{-1})}\neq0$, and $d=\frac{c+1}{a}$.
Notice that we may choose $\ell,z\in\mathbb{C}-\{\pm1\}$ so that $M_{0}\neq\ell z$ and $M_{0}\neq(\ell z)^{-1}$, and therefore $c\neq0$.
Thus, $\mathrm{tr}\rho_{2}(\alpha\beta)=M_{0}+M_{0}^{-1}$ and $\mathrm{tr}\rho_{2}(\beta)=\mathrm{tr}\rho_{1}(\lambda_{C})=\ell+\ell^{-1}$ with
$$
\rho_{2}(\lambda_{K})=\rho_{2}(\alpha^{q}\mu_{K}^{-pq})=I\cdot\rho_{2}(\mu_{K})^{-pq}=\rho_{2}(\mu_{K})^{-pq}
$$
since $\rho_{2}(\alpha^{q})=I$ by construction.
Hence, up to conjugation, we may extend $\rho_{1}$ to $\rho_{K}=\rho_{1}\ast\rho_{2}$ so that $\rho_{K}(\mu_{K})=\begin{pmatrix}M_{0}&0\\0&M_{0}^{-1}\end{pmatrix}$ and thus $\rho_{K}(\lambda_{K})=\rho_{K}(\mu_{K})^{-pq}$ for all $M_{0}\neq\pm1$, and so $\rho_{K}(\lambda_{K}\mu_{K}^{pq})=I$.
However, if $M_{0}=\pm1$, then since $\ell\neq\pm1$, $\rho_{1}(\lambda_{C})\neq\pm I$ and so $\rho_{2}(\alpha\beta)\neq\pm I$.
Since $\mathrm{tr}\rho_{2}(\mu_{K})=M_{0}+M_{0}^{-1}=\pm2$, it follows that $\rho_{2}(\mu_{K})=M_{0}\begin{pmatrix}1&1\\0&1\end{pmatrix}$ up to conjugation, and  thus $\rho_{2}(\lambda_{K})=\rho_{2}(\alpha^{q}\mu_{K}^{-pq})=\rho_{2}(\mu_{K})^{-pq}$ as before.
Therefore, $(LM^{pq}-1)$ does not have any gaps.

If $(L,M)\in\mathbb{V}(LM^{pq}+1)$ for $q\geq2$, then we construct the representation $\rho_{K}$ similarly instead using $z^{q}=-1$ so that $z^{2q}=1$, hence $\rho_{K}(\alpha^{q})=-I$.
Therefore every $(L_{0},M_{0})\in\mathbb{V}(F_{(p,q)})\cap(\mathbb{C}^{\ast})^{2}$ is realized by a representation $\rho_{K}\in R(\mathcal{M}_{K})$ with $\rho_{K}(\mu_{K})\neq\pm I$ and $\rho_{K}(\lambda_{K}\mu_{K}^{pq})=\pm I$; hence, $F_{(p,q)}|A_{[(p,q),C]}$ has no gaps.

For the factor $g_{0}=(LM^{rq^{2}}-\delta^{q})|A_{K}$ contributed by $f_{0}=(\overline{L}\,\overline{M}^{r}-\delta)|A_{C}$, recall every $(\overline{L},\overline{M})\in\mathbb{V}(\overline{L}\,\overline{M}^{r}-\delta)$ is realized by some representation $\rho_{1}\in R(\mathcal{M}_{C})$ with $\rho_{1}(\mu_{K})\neq\pm I$ and $\rho_{1}(\lambda_{C}\mu_{C}^{r})=\delta I$.
For $(L_{0},M_{0})\in\mathbb{V}(g_{0})\cap(\mathbb{C}^{\ast})^{2}$, if $M_{0}^{q}=\overline{M}\neq\pm1$, then up to conjugation, $\rho_{1}(\mu_{C})$ and $\rho_{1}(\lambda_{C})$ are diagonal, and we may extend $\rho_{1}$ to $\rho_{K}=\rho_{1}\ast\rho_{2}$ via the abelian representation $\rho_{2}$:
\begin{align*}
\rho_{2}(\lambda_{K})&=\rho_{1}(\lambda_{C})^{q}=\delta^{q}\begin{pmatrix}M_{0}^{-rq^{2}}&0\\0&M_{0}^{rq^{2}}\end{pmatrix}&\rho_{2}(\mu_{K})&=\begin{pmatrix}M_{0}&0\\0&M_{0}^{-1}\end{pmatrix}\\
\rho_{2}(\mu_{C})&=\rho_{2}(\mu_{K})^{q}=\begin{pmatrix}M_{0}^{q}&0\\0&M_{0}^{-q}\end{pmatrix}&\rho_{2}(\beta)&=\rho_{1}(\lambda_{C})=\delta\begin{pmatrix}M_{0}^{-rq}&0\\0&M_{0}^{rq}\end{pmatrix},
\end{align*}
and hence $(L_{0},M_{0})$ for $M_{0}^{q}\neq\pm1$ is realized by some representation $\rho_{K}\in R(\mathcal{M}_{K})$ with $\rho_{K}(\mu_{K})\neq\pm I$ and $\rho_{K}(\lambda_{K}\mu_{K}^{rq^{2}})=\delta^{q}I$.

Similarly, if $M_{0}=\pm1$, then for $q$ even, we have $\overline{M}=M_{0}^{q}=1$ and $L_{0}=\delta^{q}=1$, so we may take the abelian representation $\rho_{2}$ given by
\begin{align*}
\rho_{2}(\lambda_{K})&=\rho_{1}(\lambda_{C})^{q}=\begin{pmatrix}1&-rq^{2}\\0&1\end{pmatrix}&\rho_{2}(\mu_{K})&=M_{0}\begin{pmatrix}1&1\\0&1\end{pmatrix}\\
\rho_{2}(\mu_{C})&=\rho_{2}(\mu_{K})^{q}=\begin{pmatrix}1&q\\0&1\end{pmatrix}&\rho_{2}(\beta)&=\rho_{1}(\lambda_{C})=\delta\begin{pmatrix}1&-rq\\0&1\end{pmatrix}.
\end{align*}
This representation agrees with the irreducible representation $\rho_{1}\in R(\mathcal{M}_{K})$ realizing the point $(\delta,1)\in\mathbb{V}(\overline{L}\,\overline{M}^{r}-\delta)\cap(\mathbb{C}^{\ast})^{2}$, and hence $(1,\pm1)$ is realized by $\rho_{K}\in R(\mathcal{M}_{K})$ with $\rho_{K}(\mu_{K})\neq\pm I$ and $\rho_{K}(\lambda_{K}\mu_{K}^{rq^{2}})=\delta^{q}I$.
If $M_{0}=\pm1$ and $q$ is odd, notice that $\overline{M}=M_{0}^{q}$ and $\overline{L}=\delta M_{0}^{-rq}$, and since $A_{C}$ has no gaps, the point $(\delta M_{0}^{-rq},M_{0}^{q})\in\mathbb{V}(\overline{L}\,\overline{M}^{r}-\delta)$ is realized by some $\rho_{1}\in R(\mathcal{M}_{C})$ up to conjugation so that
\begin{align*}
\rho_{1}(\lambda_{C})&=\delta\rho_{1}(\mu_{C})^{-r}=\delta M_{0}^{-rq}\begin{pmatrix}1&-rq\\0&1\end{pmatrix}&\rho_{1}(\mu_{C})&=M_{0}^{q}\begin{pmatrix}1&q\\0&1\end{pmatrix}
\end{align*}
and this representation can be extended to $\rho_{K}=\rho_{1}\ast\rho_{2}\in R(\mathcal{M}_{K})$ via the abelian representation $\rho_{2}$ given by
\begin{align*}
\rho_{2}(\lambda_{K})&=\rho_{1}(\lambda_{C})^{q}=\delta^{q}M_{0}^{-rq^{2}}\begin{pmatrix}1&-rq^{2}\\0&1\end{pmatrix}&\rho_{2}(\mu_{K})&=M_{0}\begin{pmatrix}1&1\\0&1\end{pmatrix}\\
\rho_{2}(\mu_{C})&=\rho_{1}(\mu_{K})^{q}=M_{0}^{q}\begin{pmatrix}1&q\\0&1\end{pmatrix}&\rho_{2}(\beta)&=\rho_{1}(\lambda_{C})=\delta M_{0}^{-rq}\begin{pmatrix}1&-rq\\0&1\end{pmatrix}.
\end{align*}
Lastly, we consider $M_{0}\neq\pm1$ with $L_{0}=\pm1$; if $L_{0}=1$, then we may take the abelian representation $\rho_{K}\in R(\mathcal{M}_{K})$ such that $\rho_{K}(\lambda_{K})=I$ and $\rho_{K}(\mu_{K})=\begin{pmatrix}M_{0}&0\\0&M_{0}^{-1}\end{pmatrix}$.
However, if $L_{0}=-1$ and $M_{0}^{q}=\pm1$, then notice $(-1,M_{0})\in\mathbb{V}(F_{(p,q)})\cap(\mathbb{C}^{\ast})^{2}$.

For $M_{0}=\zeta\neq\pm1$ such that $\zeta^{q}=1$, notice that $q=2$ contradicts that $\zeta\neq\pm1$, hence $q>2$; in particular we have $F_{(p,q)}\doteq(LM^{pq}-1)(LM^{pq}+1)$.
Furthermore, $(-1)(\zeta)^{rq^{2}}-\delta^{q}=0$ implies that $\delta^{q}=-1$ and therefore $\delta=-1$ and $q$ is odd.
Notice that $(-1)(\zeta)^{pq}+1=0$, and hence $(-1,\zeta)\in\mathbb{V}(LM^{pq}+1)\cap(\mathbb{C}^{\ast})^{2}$.

For $M_{0}=\eta\neq\pm1$ such that $\eta^{q}=-1$, then $(-1)(\eta)^{rq^{2}}-\delta^{q}=0$ implies $\delta^{q}=(-1)(-1)^{rq}$, and so $q$ is odd.
Therefore, $\delta=(-1)^{r+1}$ and $q>2$, and so $F_{(p,q)}\doteq(LM^{pq}-1)(LM^{pq}+1)$ as before.
Since $p=\pm1$ is also odd, $(-1)(\eta)^{pq}-1=0$, and so $(-1,\eta)\in\mathbb{V}(LM^{pq}-1)\cap(\mathbb{C}^{\ast})^{2}$, and therefore, $(-1,\eta)\in\mathbb{V}(F_{(p,q)})\cap(\mathbb{C}^{\ast})^{2}$.
Hence, for every $(-1,M_{0})\in\mathbb{V}(LM^{rq^{2}}-\delta^{q})\cap(\mathbb{C}^{\ast})^{2}$ with $M_{0}^{q}=\pm1$ for $M_{0}\neq\pm1$, $(-1,M_{0})\in\mathbb{V}(F_{(p,q)})\cap(\mathbb{C}^{\ast})^{2}$.

However, $F_{(p,q)}|A_{K}$ has no gaps, so every point $(L_{0},M_{0})\in\mathbb{V}(F_{(p,q)})\cap(\mathbb{C}^{\ast})^{2}$ is realized by a representation $\rho_{K}\in R(\mathcal{M}_{K})$ such that $\rho_{K}(\mu_{K})\neq\pm I$ and $\rho_{K}(\lambda_{K}\mu_{K}^{pq})=\pm I$.
We see that such a representation will also satisfy $\rho_{K}(\lambda_{K}\mu_{K}^{rq^{2}})=\delta^{q}I$.
If $(-1,\zeta)\in\mathbb{V}(LM^{rq^{2}}-\delta^{q})$ as before, we find
$$
\rho_{K}(\lambda_{K}\mu_{K}^{rq^{2}})=\rho_{K}(\lambda_{K})\rho_{K}(\mu_{K})^{rq^{2}}=(-I)(I)^{rq}=-I=\delta^{q}I.
$$
Similarly, if $(-1,\eta)\in\mathbb{V}(LM^{rq^{2}}-\delta^{q})$ as before, we find that $r$ even implies $\delta=-1$ and $r$ odd implies $\delta=1$, so
$$
\rho_{K}(\lambda_{K}\mu_{K}^{rq^{2}})=\begin{cases}
(-I)(-I)^{rq}=-I=\delta^{q}I&:\hspace{4pt}r\text{ is even,}\\
(-I)(-I)^{rq}=I=\delta^{q}I&:\hspace{4pt}r\text{ is odd.}
\end{cases}
$$
Hence, every $(-1,\zeta)$ and $(-1,\eta)$ in $\mathbb{V}(LM^{rq^{2}}-\delta^{q})\cap(\mathbb{C}^{\ast})^{2}$ is realized by some representation $\rho_{K}\in R(\mathcal{M}_{K})$ with $\rho_{K}(\mu_{K})\neq\pm I$ such that $\rho_{K}(\lambda_{K}\mu_{K}^{rq^{2}})=\delta^{q}I$.
Since $g_{0}=LM^{rq^{2}}-\delta^{q}$ is a generic factor of $A_{K}$, $A_{K}$ has no gaps.
\end{proof}

Simple induction on $(p,q)$-cables and connected sums of torus knots mean that by Corollary~\ref{torusnogaps} and Lemmas~\ref{sumnogaps} and~\ref{cablenogaps}, we have the following theorem:
\begin{theorem}
\label{g0nogaps}
For every graph knot $K\in\mathcal{G}_{0}$, $A_{K}$ has no gaps.
\end{theorem}

By this theorem, we will be able to extend each representation $\rho_{2}^{\varepsilon}$ from the earlier discussion to a representation $\rho_{K}=\rho_{1}\ast\rho_{2}^{\varepsilon}$.
To do this, we require the following lemmas about the image of the projection map $\xi$ which considers three types of representations in $R_{U}(\mathcal{M}_{\mathrm{Sat}(P,C,f)})=R_{0}\cup R_{1}\cup R_{2}$, following the notation of Ruppe~\cite{ruppe_2016}:
\begin{align}
R_{0}&=\{\rho=\rho_{1}\ast\rho_{2}|\rho_{2}\text{ reducible}\}\\
R_{1}&=\{\rho=\rho_{1}\ast\rho_{2}|\rho_{2}\text{ irreducible and }\rho_{1}\text{ reducible}\}\\
R_{2}&=\{\rho=\rho_{1}\ast\rho_{2}|\rho_{2}\text{ irreducible and }\rho_{1}\text{ irreducible}\}.
\end{align}
Recall that our satellite space $\mathcal{M}_{2}$ has $\partial_{x}\mathcal{M}_{2}$ a torus with preferred framing $(\lambda_{x},x)=(\lambda_{K},\mu_{K})$ and $\partial_{y}\mathcal{M}_{2}$ a torus with preferred framing $(\lambda_{y},y)=(\mu_{C},\lambda_{C})$, following the gluing relation.

\begin{lemma}
\label{reps0}
Let $K=\mathrm{Sat}(P,C,f)$ be a winding number zero satellite where $\mathcal{M}_{K}=\mathcal{M}_{1}\cup_{\partial N(\ell_{y})}\mathcal{M}_{2}$ with $\mathcal{M}_{1}=\mathcal{M}_{C}$ and $\mathcal{M}_{2}=V-\overset{\circ}{N}(f(P))$, then
$$\overline{\xi(R_{0})}=\mathbb{V}(L-1).$$
\end{lemma}
\begin{proof}%[Proof of Lemma~\ref{reps0}]
Let $\rho_{1}\ast\rho_{2}\in R_{0}$, then up to conjugation, let $\rho_{2}$ be upper-triangular on $\pi_{1}(\mathcal{M}_{2})$, and since $\rho_{2}$ must have the same character as an abelian representation, we see that $\mathrm{tr}\rho_{2}(\lambda_{x})=2$ since $\lambda_{x}$ is null-homologous in $\mathcal{M}_{2}$.
Thus $L=1$ and so $\xi(\rho_{1}\ast\rho_{2})\in\mathbb{V}(L-1)$.
Considering all abelian representations with $\rho_{2}(x)=\begin{pmatrix}M&0\\0&M^{-1}\end{pmatrix}$ and $\rho_{2}(y)=I$, we find that $\rho_{2}$ extends to the trivial representation $\mathrm{id}_{1}(\pi_{1}(M_{1}))=\{I\}$.
Therefore, $\mathrm{id}_{1}\ast\rho_{2}\in R_{0}$ and $\xi(\mathrm{id}_{1}\ast\rho_{2})=(1,M)$ for all $M\in\mathbb{C}^{\ast}$.
Hence, $\overline{\xi(R_{0})}=\mathbb{V}(L-1)$.
\end{proof}
\begin{lemma}
\label{reps01}
Let $K=\mathrm{Sat}(P,C,f)$ be a winding number zero satellite where $\mathcal{M}_{K}=\mathcal{M}_{1}\cup_{\partial N(\ell_{y})}\mathcal{M}_{2}$ with $\mathcal{M}_{1}=\mathcal{M}_{C}$ and $\mathcal{M}_{2}=V-\overset{\circ}{N}(f(P))$, then
$$\overline{\xi(R_{0})}\cup\overline{\xi(R_{1})}=\mathbb{V}(A_{P}).$$
\end{lemma}
\begin{proof}%[Proof of Lemma~\ref{reps01}]
By Lemma~\ref{reps0} and Remark~\ref{trivfactor}, $\overline{\xi(R_{0})}=\mathbb{V}(L-1)\subset\mathbb{V}(A_{P})$.
Let $\rho_{1}\ast\rho_{2}\in R_{1}$ and up to conjugation, let $\rho_{1}$ be lower-triangular (since $\rho_{1}$ is reducible).
Since $\rho_{1}\in R(\mathcal{M}_{1})$ is reducible, $\rho_{1}(\lambda_{C})=I$ and thus $\rho_{2}(y)=I$ by the gluing relation.
Hence, we let $\mathcal{M}_{P}=V(1/0)-\overset{\circ}{N}(f(P))$ be the quotient of $\mathcal{M}_{2}$ by $(1/0)$-Dehn filling along $\partial_{y}\mathcal{M}_{2}$.
The quotient map $Q_{0}:\mathcal{M}_{2}\to\mathcal{M}_{P}$ induces an epimorphism $Q_{0\ast}:\pi_{1}(\mathcal{M}_{2})\to\pi_{1}(\mathcal{M}_{P})$ satisfying $Q_{0\ast}(y)=e$.
Since $Q_{0\ast}$ is surjective, $\rho_{2}$ factors through the quotient; that is, there is an irreducible representation $\sigma\in R(\mathcal{M}_{P})$ such that $\rho_{2}=\sigma\circ Q_{0\ast}$ and $\sigma(Q_{0\ast}(y))=I$.
Hence, $\xi(\rho_{1}\ast\rho_{2})=\xi(\sigma)$ since $Q_{0\ast}(x)=\mu_{P}$ and $Q_{0\ast}(\lambda_{x})=\lambda_{P}$, and $\xi(\rho_{1}\ast\rho_{2})$ is either an isolated point or in a component $R_{f_{0}}\subset R(\mathcal{M}_{P})$.
In the latter case, we find that $\xi(\rho_{1}\ast\rho_{2})=\xi(\sigma)\in\mathbb{V}(A_{P})$, and therefore $\overline{\xi(R_{0})}\cup\overline{\xi(R_{1})}\subset\mathbb{V}(A_{P})$.
Note that the isolated points $\xi(\sigma)$ will only lift to isolated points in $\xi(R_{1})$ and so no other factors will appear.

For any $\sigma\in R(\mathcal{M}_{P})$ with $\xi(\sigma)\in\mathbb{V}(f_{0})$ for a balanced-irreducible factor $f_{0}|\widetilde{A}_{P}$, then the representation $\sigma$ will lift to some representation $\rho_{2}\in R(\mathcal{M}_{2})$ satisfying $\rho_{2}(y)=I$ from the quotient map $\rho_{2}=\sigma\circ Q_{0\ast}$.
$$
\begin{tikzcd}
\pi_{1}(\mathcal{M}_{2})\arrow[d, "Q_{0\ast}"]\arrow[dr, "\rho_{2}"]\\
\pi_{1}(\mathcal{M}_{P})\arrow[r, "\sigma"]&SL_{2}\mathbb{C}
\end{tikzcd}
$$
Conjugating $\rho_{2}$ so that $\rho_{2}$ is lower-triangular on $\partial_{y}\mathcal{M}_{2}$, we may take an abelian representation $\rho_{1}\in R(\mathcal{M}_{1})$ to send $\rho_{1}(\lambda_{C})=I$ and $\rho_{1}(\mu_{C})=\rho_{2}(\lambda_{y})$.
Therefore, we have a representation $\rho_{1}$ which agrees with $\rho_{2}$ along the gluing boundary and so $\rho_{1}\ast\rho_{2}\in R_{1}$, and up to conjugation, the representation $\rho_{1}\ast\rho_{2}$ can be made upper-triangular on $\partial_{x}\mathcal{M}_{2}$ with $\xi(\rho_{1}\ast\rho_{2})=\xi(\sigma)\in\mathbb{V}(\widetilde{A}_{P})$, which completes the proof, $\overline{\xi(R_{0})}\cup\overline{\xi(R_{1})}=\mathbb{V}(A_{P})$.
\end{proof}
\begin{lemma}
\label{reps012}
Let $K=\mathrm{Sat}(P,C,f)$ be a winding number zero satellite with companion knot $C\in\mathcal{G}_{\mathbb{Z}}$.
Let $\mathcal{M}_{K}=\mathcal{M}_{1}\cup_{\partial N(\ell_{y})}\mathcal{M}_{2}$ with $\mathcal{M}_{1}=\mathcal{M}_{C}$ and $\mathcal{M}_{2}=V-\overset{\circ}{N}(f(P))$, and let $f(P)_{r}$ be the knot whose exterior is given by $V(1/r)-\overset{\circ}{N}(f(P))$, then
$$
\overline{\xi(R_{0})}\cup\overline{\xi(R_{1})}\cup\overline{\xi(R_{2})}\subset\mathbb{V}\left(\mathrm{Red}\left[(L-1)\underset{r\in\mathcal{DS}_{C}}{\prod}\widetilde{A}_{f(P)_{r}}\right]\right).
$$
\end{lemma}
\begin{proof}%[Proof of Lemma~\ref{reps012}]
By Lemma~\ref{reps01}, we know that $\overline{\xi(R_{0})}\cup\overline{\xi(R_{1})}=\mathbb{V}(A_{P})=\mathbb{V}\left((L-1)\widetilde{A}_{f(P)_{0}}\right)$ by definition of $f(P)_{0}$, so these factors will appear in the variety on the right.

Let $\rho_{1}\ast\rho_{2}\in R_{2}$, then since $C\in\mathcal{G}_{\mathbb{Z}}$ we may assume that $\rho_{1}\in R^{\ast}(\mathcal{M}_{1})$ satisfies $\rho_{1}(\lambda_{C}\mu_{C}^{r})=\delta I$ for some slope $r\in\mathcal{DS}_{C}$ and $\delta\in\{\pm1\}$.
Then $\rho_{2}(y\lambda_{y}^{r})=\delta I$ by the gluing relation and up to conjugation we may take $\rho_{2}(y)$ to be lower-triangular,
\begin{align*}
\rho_{2}(y)&=\begin{pmatrix}
u&0\\t&u^{-1}
\end{pmatrix}&
\rho_{2}(\lambda_{y})&=\begin{pmatrix}
v&0\\s&v^{-1}
\end{pmatrix},
\end{align*}
where $\rho_{2}(y\lambda_{y}^{r})=\delta I$ by the gluing relation.
The quotient $Q_{r}:\mathcal{M}_{2}\to\mathcal{M}_{f(P)_{r}}$ by $(1/r)$-Dehn filling along $\partial_{y}\mathcal{M}_{2}$ induces the map $Q_{r\ast}:\pi_{1}(\mathcal{M}_{2})\to\pi_{1}(\mathcal{M}_{f(P)_{r}})$ which satisfies $Q_{r\ast}(y{\lambda_{y}}^{r})=e$ and so letting $\varepsilon_{2}:\pi_{1}(\mathcal{M}_{2})\to\{\pm I\}$ be the abelian representation given by $\varepsilon_{2}(y)=\delta I$ and $\varepsilon_{2}(x)=I$, we find that
$\rho^{\varepsilon_{2}}_{2}=\varepsilon_{2}\cdot\rho_{2}$ is a representation of $\pi_{1}(\mathcal{M}_{2})$ satisfying
$$
\rho^{\varepsilon_{2}}_{2}(y\lambda_{y}^{r})=\varepsilon_{2}(y\lambda_{y}^{r})\cdot\rho_{2}(y{\lambda_{y}}^{r})=\delta I\cdot \delta I=I.
$$

Since $Q_{r\ast}$ is surjective (as in the proof of Lemma~\ref{reps01}), there is some irreducible $\sigma\in R(\mathcal{M}_{f(P)_{r}})$ such that $\rho^{\varepsilon_{2}}_{2}=\sigma\circ Q_{r\ast}$.
$$
\begin{tikzcd}
\pi_{1}(\mathcal{M}_{2})\arrow[d, "Q_{r\ast}"]\arrow[dr, "\rho^{\varepsilon_{2}}_{2}"]\\
\pi_{1}(\mathcal{M}_{f(P)_{r}})\arrow[r, "\sigma"]&SL_{2}\mathbb{C}
\end{tikzcd}
$$
Therefore, $\xi(\rho_{1}\ast\rho_{2})=\xi(\sigma)$ since $Q_{r\ast}(x)=\mu_{f(P)_{r}}$ and $Q_{r\ast}(\lambda_{x})=\lambda_{f(P)_{r}}$, and each $\xi(\sigma)$ is either an isolated point or in a component contributing a factor of $\mathbb{V}(A_{f(P)_{r}})$.
As with the earlier proof, the isolated points will only lift to isolated points, but for every $\xi(\sigma)$ in a component of $\mathbb{V}(A_{f(P)_{r}})$, the lifted point $\xi(\rho_{1}\ast\rho_{2})$ will still be in a component of $\mathbb{V}(A_{f(P)_{r}})$, and therefore
$$
\overline{\xi(R_{0})}\cup\overline{\xi(R_{1})}\cup\overline{\xi(R_{2})}\subset\mathbb{V}\left(\mathrm{Red}\left[(L-1)\prod_{r\in\mathcal{DS}_{C}}\widetilde{A}_{f(P)_{r}}\right]\right).
$$
\end{proof}
As an immediate consequence of these lemmas, we can find a polynomial multiple of the $A$-polynomial of such winding number zero satellite knots where the companion knot $C\in\mathcal{G}_{\mathbb{Z}}$:
\begin{theorem}
Let $K=\mathrm{Sat}(P,C,f)$ be a winding number zero satellite knot with companion knot $C\in\mathcal{G}_{\mathbb{Z}}$, let $\mathcal{M}_{K}=\mathcal{M}_{1}\cup_{\partial N(\ell_{y})}\mathcal{M}_{2}$ where $\mathcal{M}_{1}=\mathcal{M}_{C}$ and $\mathcal{M}_{2}=V-\overset{\circ}{N}(f(P))$, and let $f(P)_{r}$ be the knot whose exterior is given by $V(1/r)-\overset{\circ}{N}(f(P))$, then
$$
A_{K}\Bigg|\mathrm{Red}\left[(L-1)\prod_{r\in\mathcal{DS}_{C}}\widetilde{A}_{f(P)_{r}}\right].
$$
\end{theorem}

To show that each of the factors on the right is a factor of $A_{K}$, we will utilize Theorem~\ref{g0nogaps}, that is, that the $A$-polynomial of a graph knot has no gaps.

\begin{lemma}
\label{reps012d}
Let $K=\mathrm{Sat}(P,C,f)$ be a winding number zero satellite with companion knot $C\in\mathcal{G}_{0}$.
Let $\mathcal{M}_{K}=\mathcal{M}_{1}\cup_{\partial N(\ell_{y})}\mathcal{M}_{2}$ with $\mathcal{M}_{1}=\mathcal{M}_{C}$ and $\mathcal{M}_{2}=V-\overset{\circ}{N}(f(P))$, and let $f(P)_{r}$ be the knot whose exterior is given by $V(1/r)-\overset{\circ}{N}(f(P))$, then
$$
\overline{\xi(R_{0})}\cup\overline{\xi(R_{1})}\cup\overline{\xi(R_{2})}=\mathbb{V}\left(\mathrm{Red}\left[(L-1)\underset{r\in\mathcal{DS}_{C}}{\prod}\widetilde{A}_{f(P)_{r}}\right]\right).
$$
\end{lemma}
\begin{proof}
By Lemma~\ref{reps012}, it suffices to show the other direction of containment.
Notice that Lemma~\ref{reps01} implies that $\overline{\xi(R_{0})}\cup\overline{\xi(R_{1})}=\mathbb{V}((L-1)\widetilde{A}_{f(P)_{0}})$.

Therefore, let $\widetilde{A}_{f(P)_{r}}$ be a factor with corresponding $f_{0}=(LM^{r}-\delta)|\widetilde{A}_{C}$, so $r\in\mathcal{DS}_{C}$ and $\delta\in\{\pm1\}$ are given.
Let $\mathcal{M}_{f(P)_{r}}$, $Q_{r}$, and $Q_{r\ast}$ be as in the previous proof, and so $Q_{r\ast}(\lambda_{y}^{r}y)=e$.
For each balanced-irreducible factor $g_{0}|\widetilde{A}_{f(P)_{r}}$, there is a family of representations $\sigma\in R(\mathcal{M}_{f(P)_{r}})$ such that $\xi(\sigma)=(L,M)\in\mathbb{V}(g_{0})$ for all but finitely many points.
Let $\rho_{2}=\sigma\circ Q_{r\ast}$ be the lift of such a representation, as in the proof of Lemma~\ref{reps012}, then to find a representation $\rho^{\varepsilon_{2}}_{2}\in R^{\ast}(\mathcal{M}_{2})$ which agrees with some $\rho_{1}\in R^{\ast}(\mathcal{M}_{1})$ along the gluing torus, we use the same abelian representation $\varepsilon_{2}$ from the proof of Lemma~\ref{reps012}, $\varepsilon_{2}:\pi_{1}(\mathcal{M}_{2})\to\{\pm I\}$ given by $\varepsilon(y)=\delta I$ and $\varepsilon(x)=I$.
This gives $\rho^{\varepsilon_{2}}_{2}=\varepsilon_{2}\cdot\rho_{2}$ satisfying $\rho^{\varepsilon_{2}}_{2}(y\lambda_{y}^{r})=\delta I$.
Additionally, we note that $\rho^{\varepsilon_{2}}_{2}(\lambda_{y})=\rho_{2}(\lambda_{y})$ and $\rho^{\varepsilon_{2}}_{2}(\lambda_{x})=\rho_{2}(\lambda_{x})$ since $\mathcal{M}_{2}$ is a winding number zero satellite space.
We show now that every $\rho^{\varepsilon_{2}}_{2}\in R(\mathcal{M}_{2})$ from this family of representations $\sigma\in R^{\ast}(\mathcal{M}_{f(P)_{r}})$ will extend to a representation $\rho_{1}\in R(\mathcal{M}_{C})$ since $A_{C}$ has no gaps.

If $\mathrm{tr}\rho^{\varepsilon_{2}}_{2}(\lambda_{y})\neq\pm2$, then we may conjugate $\rho^{\varepsilon_{2}}_{2}$ so that $\rho^{\varepsilon_{2}}_{2}(\lambda_{y})=\begin{pmatrix}\overline{M}&0\\0&\overline{M}^{-1}\end{pmatrix}$ and thus by the quotient identity, $\rho^{\varepsilon}_{2}(y)=\delta\rho^{\varepsilon}_{2}(\lambda_{y}^{-r})=\delta\begin{pmatrix}\overline{M}^{-r}&0\\0&\overline{M}^{r}\end{pmatrix}$.
Since $C\in\mathcal{G}_{0}$ and $A_{C}$ has no gaps by Theorem~\ref{g0nogaps}, there is a representation $\rho_{1}\in R(\mathcal{M}_{1})$ such that $\rho_{1}(\mu_{C})=\begin{pmatrix}\overline{M}&0\\0&\overline{M}^{-1}\end{pmatrix}$ and $\rho_{1}(\lambda_{C})=\delta\begin{pmatrix}\overline{M}^{-r}&0\\0&\overline{M}^{r}\end{pmatrix}$.
Therefore the representation $\rho^{\varepsilon_{2}}_{2}$ extends to a representation $\rho_{K}=\rho_{1}\ast\rho^{\varepsilon_{2}}_{2}\in R(\mathcal{M}_{K})$ with $\xi(\rho_{K})=\xi(\sigma)$.

If $\mathrm{tr}\rho^{\varepsilon_{2}}_{2}(\lambda_{y})=\pm2=2\overline{M}$, then up to conjugation, either $\rho^{\varepsilon_{2}}_{2}(\lambda_{y})=\overline{M} I$ or $\rho^{\varepsilon_{2}}_{2}=\overline{M}\begin{pmatrix}1&1\\0&1\end{pmatrix}$.
In the latter case, the gluing relation implies that
\begin{align*}
\rho^{\varepsilon_{2}}_{2}(\lambda_{y})&=\overline{M}\begin{pmatrix}
1&1\\
0&1
\end{pmatrix},
&
\rho^{\varepsilon_{2}}_{2}(\mu_{y})&=\delta\overline{M}^{-r}\begin{pmatrix}
1&-r\\
0&1
\end{pmatrix}.
\end{align*}
Since $(\delta\overline{M}^{-r},\overline{M})\in\mathbb{V}(LM^{r}-\delta)\cap(\mathbb{C}^{\ast})^{2}$, there is a representation $\rho_{1}\in R(\mathcal{M}_{C})$ such that $\rho_{1}(\lambda_{C}\mu_{C}^{r})=\delta I$ and $\rho_{1}(\mu_{C})\neq\pm I$ with $\mathrm{tr}\rho_{1}(\mu_{C})=2\overline{M}$.
Therefore, the representation will extend to $\rho_{K}=\rho_{1}\ast\rho^{\varepsilon_{2}}_{2}\in R(\mathcal{M}_{K})$ such that $\xi(\rho_{K})=\xi(\sigma)$.
In the former case, $\rho^{\varepsilon_{2}}_{2}(\lambda_{y})=\pm I$, then we use some abelian representation $\varepsilon$ so that $\rho^{\varepsilon}_{2}(y)=I$, which will naturally extend to $\rho_{K}=\varepsilon_{1}\ast\rho^{\varepsilon}_{2}$ via the abelian representation $\varepsilon_{1}:\pi_{1}(\mathcal{M}_{C})\to\{\pm I\}$ given by $\varepsilon_{1}(\mu_{C})=\rho^{\varepsilon}_{2}(\lambda_{y})$.
Therefore, every representation $\sigma\in R^{\ast}(\mathcal{M}_{f(P)_{r}})$ will extend to some $\rho_{K}\in R(\mathcal{M}_{K})$ with $\xi(\rho_{K})=\xi(\sigma)$; hence each factor $\mathbb{V}(g_{0})\subset\mathbb{V}(\widetilde{A}_{K})$ and therefore the lemma is proven.
\end{proof}

We see that Theorem~\ref{wnzthm} follows from Lemmas~\ref{reps012} and~\ref{reps012d} which will be discussed in Section~\ref{proof2}.
Furthermore, Theorems~\ref{doublellin} and ~\ref{doubletwisteddouble} follow as a consequence of Theorem~\ref{wnzthm}, where we omit polynomial reduction of $A_{D_{n}(C)}$ for $C\in\mathcal{G}_{0}$ because each factor $\widetilde{A}_{f(P)_{r}}=\widetilde{A}_{K(n-r)}$ is irreducible and distinct, also discussed in Section~\ref{proof2}.

%%%%%%%%%%%%%%%%%%%%%%%%%%%%%%%%%%%%%%%%%%%%%%%%%%
% SECTION 6: r-Twisted WHD Facts for Computation %
%%%%%%%%%%%%%%%%%%%%%%%%%%%%%%%%%%%%%%%%%%%%%%%%%%
\section{$r$-Twisted Gluing Relations}
\label{rtwistedcomps}
In the special case of $r$-twisted Whitehead doubles, the subset of importance in $R(\mathcal{M}_{D_{r}(K)})$ is $\overline{\xi(R_{2})}$ given by representations $\rho=\rho_{1}\ast\rho_{2}$ where both $\rho_{1},\rho_{2}$ are irreducible representations since Remark~\ref{patternfactor} gives us the known factor $A_{K(r)}|A_{D_{r}(K)}$ for any $r$-twisted Whitehead double.

For explicit computation when both $\rho_{1},\rho_{2}$ are irreducible, we may conjugate $\rho_{2}$ so that
\begin{align*}
\rho_{2}(\mu_{x})&=\begin{pmatrix}
M&1\\
0&M^{-1}
\end{pmatrix}=\rho_{2}(x),&
\rho_{2}(\mu_{y})&=\begin{pmatrix}
u&0\\
t&u^{-1}
\end{pmatrix}=\rho_{2}(y).
\end{align*}
And since $\mu_{x},\lambda_{x}$ commute and $\mu_{y},\lambda_{y}$ commute, we have
\begin{align*}
\rho_{2}(\lambda_{x})&=\begin{pmatrix}
L&\ast\\
0&L^{-1}
\end{pmatrix},&
\rho_{2}(\lambda_{y})&=\begin{pmatrix}
v&0\\
s&v^{-1}
\end{pmatrix}\overset{\phi_{r}}{=}\rho_{1}(\mu_{C})
=\begin{pmatrix}
m&0\\
\ast&m^{-1}
\end{pmatrix}.
\end{align*}
Furthermore, since $\mu_{K},\lambda_{K}$ commute, and $\mu_{K}$ is lower-triangular by the gluing, this implies $\lambda_{K}$ is also lower triangular, and label its $(1,1)$-entry $\ell$.
Here, we use the $r$-twisted gluing relation $\phi_{r}:\partial\mathcal{M}_{K}\to\partial V$, given by
\begin{align*}
\phi_{r}(\mu_{K})&=\lambda_{y},&\phi_{r}(\lambda_{K})&=\mu_{y}\lambda_{y}^{-r}.
\end{align*}
Hence, we have $\rho_{2}(y)=\rho_{1}(\lambda_{K})$ and $\rho_{2}(\mu_{y}\lambda_{y}^{-r})=\rho_{1}(\lambda_{K})$, whose $(1,1)$-entry gives us an additional relation on $\ell$; these combined give us:
\begin{align*}
m&=v,&
\ell&=\rho_{1}(\lambda_{K})_{1,1}=\rho_{2}(y\lambda_{y}^{-r})_{1,1}=uv^{-r}.
\end{align*}
Hence, if $\rho\in R_{2}$, then the (1,1)-entries of $\rho_{1}(\mu_{K})$ and $\rho_{1}(\lambda_{K})$ must satisfy $\widetilde{A}_{K}\left(\ell,m\right)=0$, or alternatively
\begin{align}
f_{K,r}(M,t,u)&=\widetilde{A}_{K}\left(uv^{-r},v\right)=0,
\end{align}
\begin{align}
v&=\tfrac{-M t^{2}+M^{3} t^{2}-t u+2 M^{2} t u-M^{4} t u+M^{2} t^{3} u+M u^{2}+M t^{2} u^{2}-2 M^{3} t^{2} u^{2}-M^{2} t u^{3}+M^{4} t u^{3}}{M u^{2}},
\end{align}
\begin{align}
s&=\rho_{2}(\lambda_{y})_{2,1}.
\end{align}
The Whitehead relation gives us $\rho_{2}(\Omega)=\rho_{2}(\Omega^{\ast})$ which is true so long as a single polynomial equation is satisfied:
\begin{align}
f_{W}(M,t,u)=\begin{matrix}M^2 t - M^4 t - M u + M^3 u \\- M t^2 u + 2 M^3 t^2 u + t u^2 - 4 M^2 t u^2 \\+ M^4 t u^2 - M^2 t^3 u^2 + M u^3 - M^3 u^3 \\+ 2 M t^2 u^3 - M^3 t^2 u^3 - t u^4 + M^2 t u^4\end{matrix}=0.
\end{align}
Lastly, $\rho_{2}(\lambda_{x})_{1,1}=\rho_{2}(XY\Omega YX)_{1,1}$ gives us an additional polynomial equation:
\begin{align}
F_{W}(L,M,t,u)=\begin{matrix}M t - M^{3} t - t^{2} u + 2 M^{2} t^{2} u - 2 M t u^{2} + M^{3} t u^{2} \\
- M t^{3} u^{2} - M^{2} u^{3} + L M^{2} u^{3} + t^{2} u^{3} - M^{2} t^{2} u^{3} + M t u^{4}\end{matrix}=0.
\end{align}
Keeping as many of these defining equations constant as possible is the reason for the choice of the $r$-twisted gluing $\phi_{r}$ with the same Whitehead link $W$.
From this, we see that if $\overline{\xi(R_{2})}$ contributes a factor $\widetilde{P}_{K,r}$ of the $A$-polynomial, then its variety $\mathbb{V}(\widetilde{P}_{K,r})\subset\mathbb{V}(\mathrm{Res}_{u,t}(f_{K,r},f_{W},F_{W}))$.
From these three polynomials, we are able to perform resultant methods to find (by explicit computation) a polynomial which contains the $\widetilde{P}_{K,r}$ as a factor: $\mathrm{Res}_{u}\left[\mathrm{Res}_{t}[f_{K,r},f_{W}],\mathrm{Res}_{t}[f_{W},F_{W}]\right]$.
Removing isolated points and impossible factors (since $M\neq0$, $u\neq0$, $t\neq0$, etc) as well as checking against possible boundary slopes, we may eliminate incorrect factors from this ``iterated resultant.''
The connection between $A_{D_{r}(K)}$ and the $A$-polynomial of $n$-twist knots is given clearly by Theorem~\ref{doublellin}; a recursive formula for $A_{K(n)}$ was first found by Hoste and Shanahan~\cite{hs_2004}, and later an explicit formula by Mathews~\cite{mathews_2014}:
\begin{theorem}\cite{mathews_2014}
For any $n$-twist knot $K(n)$, its $A$-polynomial is given explicitly as:
\label{twistpolynomial}
$$
\textstyle
\widetilde{A}_{K(n)}=\begin{cases}
M^{2n}(L+M^{2})^{2n-1}\times\\
\hspace{10pt}\times\underset{i=0}{\overset{2n-1}{\sum}}\binom{n+\left\lfloor\frac{i-1}{2}\right\rfloor}{i}\left(\tfrac{M^{2}-1}{L+M^{2}}\right)^{i}(1-L)^{\left\lfloor\frac{i}{2}\right\rfloor}(M^{2}-\tfrac{L}{M^{2}})^{\left\lfloor\frac{i+1}{2}\right\rfloor}
&:\hspace{4pt}n\geq0\\
M^{-2n}(L+M^{2})^{-2n}\times\\
\hspace{10pt}\times\underset{i=0}{\overset{-2n}{\sum}}\binom{-n+\left\lfloor\frac{i}{2}\right\rfloor}{i}\left(\tfrac{1-M^{2}}{L+M^{2}}\right)^{i}(1-L)^{\left\lfloor\frac{i}{2}\right\rfloor}(M^{2}-\tfrac{L}{M^{2}})^{\left\lfloor\frac{i+1}{2}\right\rfloor}
&:\hspace{4pt}n\leq0.
\end{cases}
$$
\end{theorem}
\begin{theorem}\cite{hs_2004}
\label{twistpolynomialrec}
For any $n$-twist knot $K(n)$, its $A$-polynomial is given recursively as:
$$
\widetilde{A}_{K(n)}=\begin{cases}
x\widetilde{A}_{K\left(n-\tfrac{n}{|n|}\right)}-y\widetilde{A}_{K\left(n-\tfrac{2n}{|n|}\right)}&:\hspace{4pt}n\neq-1,0,1,2\\
M^{4}+L(-1+M^{2}+2M^{4}+M^{6}-M^{8})+L^{2}M^{4}&:\hspace{4pt}n=-1\\
1&:\hspace{4pt}n=0\\
L+M^{6}&:\hspace{4pt}n=1\\
M^{14}+L(M^{4}-M^{6}+2M^{10}+2M^{12}-M^{14})\\+L^{2}(-1+2M^{2}+2M^{4}-M^{8}+M^{10})+L^{3}&:\hspace{4pt}n=2
\end{cases}
$$
where
\begin{align*}
x&=L^{2}(M^{4}+1)+L(-M^{8}+2M^{6}+2M^{4}+2M^{2}-1)+M^{4}\\
&=(L+M^{2})\widetilde{A}_{K(1)}+\widetilde{A}_{K(-1)}\\
y&=M^{4}(L+M^{2})^{4}.
\end{align*}
\end{theorem}
Note that Hose and Shanahan's convention actually gives $\widetilde{A}_{K(n)^{\ast}}$ under the notation in this paper; to remedy this, the mirror image is found by Remark \ref{mirrored}, which will not matter for $K(-1)$.
In general, $K(n)^{\ast}=J(2,2n)^{\ast}=J(-2,-2n)=J(-2n,2)$.

Further examples of winding number zero satellites can be described with links where $V=\mathcal{M}_{\ell_{y}}$ and so $f(P)=\ell_{x}\subset V$.
The first generalization of the $r$-twisted Whitehead doubles we consider are the $(m,n)$-double twisted doubles, whose pattern knot embedding is shown here with $m$ vertical full-twists and $n$ vertical full-twists:
\begin{figure}[th]
$$
\begin{tikzpicture}
		\begin{knot}[
%		draft mode=crossings,
		line width=1.5pt,
		line join=round,
		clip width=1,
		scale=2,
		background color=white,
		consider self intersections,
		only when rendering/.style={
			draw=white,
			double=black,
			double distance=1.5pt,
			line cap=none
		}
		]
\strand
(.85,.6) to ++(-.3,0);
\strand
(-.35,.5) to ++(.3,0);
\strand
(-.2,.5) arc (0:90:.15)
to ++(-.2,0)
arc (90:180:.15)
to ++(0,-.8);
\strand
(0,0) to +(.4,0)
arc (-90:0:.3)
to +(0,1.2)
arc (0:90:.3)
to +(-.4,0)
arc (-270:-180:.15)
.. controls +(0,-.15) and +(0,.15) .. ++(-.3,-.3)
.. controls +(0,-.15) and +(0,.15) .. ++(.3,-.3)
arc (-180:-90:.15)
to +(.25,0)
arc (90:0:.15)
to +(0,-.3)
arc (0:-90:.15)
to ++(-1.1,0)
arc (-90:-180:.15)
to ++(0,.3)
arc (180:90:.15)
to ++(.25,0)
arc (-90:0:.15)
.. controls +(0,.15) and +(0,-.15) .. ++(.3,.3)
.. controls +(0,.15) and +(0,-.15) .. ++(-.3,.3)
arc (0:90:.15)
to ++(-.4,0)
arc (90:180:.3)
to ++(0,-1.2)
arc (-180:-90:.3)
to ++(1,0);
\strand
(-.2,.5) to ++(0,-.8)
arc (0:-90:.15)
to ++(-.2,0)
arc (-90:-180:.15);
\flipcrossings{1,2}
\end{knot}
\draw[very thick, -latex] (1.1,1.2) to ++(-.1,0);
\draw[very thick, -latex] (-.7,1) to ++(-.1,0);
\draw[thick,fill=white] (-2.8,.9) -- ++(1,0) -- ++(0,.6) -- ++(-1,0) -- ++(0,-.6) -- cycle;
\draw[thick,fill=white] (-1.1,2.1) -- ++(1,0) -- ++(0,1.2) -- ++(-1,0) -- ++(0,-1.2) -- cycle;
\node[right] at (1.8,1.2) {$x$};
\node[right] at (0,1) {$y$};
\node at (1.7,2) {$\ell_{x}$};
\node at (-.1,-.4) {$\ell_{y}$};
\node at (-2.3,1.2) {$n$};
\node at (-.6,2.7) {$m$};
\end{tikzpicture}
$$
\vspace*{0pt}
\caption{The $(m,n)$-Double Twisted Double Pattern.
\label{fig2}}
\end{figure}
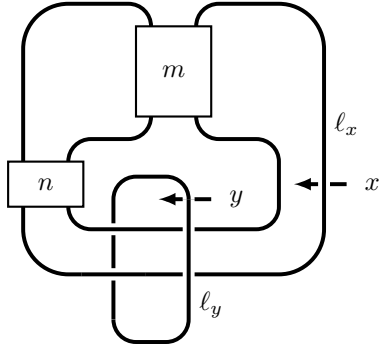

%%%%%%%%%%%%%&&&&&&&&&&&%%%%%%%%%%%%%%%%%%%%%%%%%%
% Section 7: Proofs of Whitehead Double Theorems %
%%%%%%%%%%%%%&&&&&&&&&&&%%%%%%%%%%%%%%%%%%%%%%%%%%
\section{Proof of Theorems~\ref{wnzthm} and~\ref{doublellin}}
\label{proof2}
We prove the more general result about the $r$-twisted Whitehead doubles of graph knots, then refer to specific examples as corollaries of .
We begin with a theorem from Ruppe's thesis:
\begin{theorem}~\cite{ruppe_2016}
\label{ruppethm}
For a $(p,q)$-torus knot $T(p,q)$,
\begin{align}
A_{D_{r}(T(p,q))}=(L-1)\widetilde{A}_{K(r)}\widetilde{A}_{K(r-pq)}.
\end{align}
\end{theorem}
As one possible generalization of this result, we will show for an iterated torus knot $[(p_{1},q_{1}),\ldots,(p_{n},q_{n})]$, the $A$-polynomial of the  $r$-twisted Whitehead double of this knot is given by Corollary~\ref{iteratedtorusdouble}; another generalization is the $r$-twisted Whitehead double of the connected sum of two torus knots, given in Corollary~\ref{doubleconnectedsum}.
In general, for a graph knot $C\in\mathcal{G}_{0}$, Theorem~\ref{doublellin} will give the $A$-polynomial of any $n$-twisted Whitehead double $A_{D_{n}(C)}$.
Notice in Theorem~\ref{ruppethm}, the first two factors of $A_{D_{r}(T(p,q))}$ are $(L-1)$ and $\widetilde{A}_{K(r)}$, which is a restatement of the fact that $A_{P}|A_{\mathrm{Sat}(P,C,f)}$.
Recalling the earlier notation $A$-polynomial of the satellite knot $\mathrm{Sat}(P,C,f)$, for embedded pattern knot $P$ and companion knot $C$, can be written in the form:
$$
A_{\mathrm{Sat}(P,C,f)}=\mathrm{Red}[A_{P}\widetilde{F}_{\mathrm{Sat}(P,C,f)}].
$$
For $P=K(r)$, the last factor $\widetilde{F}_{\mathrm{Sat}(K(r),C,f)}|A_{\mathrm{Sat}(K(r),C,f)}=A_{D_{r}(C)}$ that requires Theorem~\ref{wnzthm}, which is a consequence of the results from Section~\ref{zerodouble}:
\subsection*{Proof of Theorem~\ref{wnzthm}}
By Lemmas~\ref{reps0},~\ref{reps01}, and~\ref{reps012}, for a winding number zero satellite knot $K=\mathrm{Sat}(P,C,f)$ with $C\in\mathcal{G}_{0}$, we find that
$$
\overline{\xi(R_{U}(\mathcal{M}_{K}))}=\mathbb{V}\left((L-1)\underset{r\in\mathcal{DS}_{C}}{\prod}\widetilde{A}_{f(P)_{r}}\right).
$$
By Lemma~\ref{reps012d}, we see that $A_{f(P)_{r}}|A_{K}$ for each slope $r\in\mathcal{DS}_{C}$, and hence
$$
\mathrm{Red}\left[(L-1)\underset{r\in\mathcal{DS}_{C}}{\prod}\widetilde{A}_{f(P)_{r}}\right]\Big|A_{K}.
$$
Furthermore, for all but finitely many points $(L,M)$ in the zero locus of the product, we find that there will be representations $\rho\in R_{U}(\mathcal{M}_{K})$ such that $\xi(\rho)=(L,M)$ and hence
$$
A_{K}=\mathrm{Red}\left[(L-1)\underset{r\in\mathcal{DS}_{C}}{\prod}\widetilde{A}_{f(P)_{r}}\right].
$$
We must reduce this polynomial formula in general since the $A_{K}$ does not contain any repeated factors, and depending on the different $A$-polynomials $\widetilde{A}_{f(P)_{r}}$, there may be repeated factors.
This completes the proof.
\null\hfill$\square$\\\\

Some additional lemmas to omit polynomial reduction for $A_{D_{n}(C)}$ are:
\begin{lemma}~\cite{hs_2004}
\label{irreducibletwists}
For any $n$-twist knot $K(n)$, $\widetilde{A}_{K(n)}$ is irreducible.
\end{lemma}
\begin{lemma}~\cite{hs_2004}
\label{uniquetwists}
For two integers $m\neq n$, $A_{K(m)}\neq A_{K(n)}$.
\end{lemma}
\subsection*{Proof of Theorem~\ref{doublellin}}
By Lemma~\ref{reps012d}, we know that
$$
\overline{\xi(R_{0})}\cup\overline{\xi(R_{1})}\cup\overline{\xi(R_{2})}=\mathbb{V}\left((L-1)\underset{r\in\mathcal{DS}_{C}}{\prod}\widetilde{A}_{f(K(n))_{r}}\right),
$$
where $f(K(n))_{r}$ is the knot obtained from the embedding of $K(n)$ into the solid torus given by Figure~\ref{fig1} with the $(1/r)$-Dehn filling $V(1/r)$.
Notice that $f(K(n))_{r}=K(n-r)$ since the $(1/r)$-Dehn filling can be understood as $-r$ full-twists on the boundary of $V$, hence,
$$
\overline{\xi(R_{0})}\cup\overline{\xi(R_{1})}\cup\overline{\xi(R_{2})}=\mathbb{V}\left((L-1)\prod_{r\in\mathcal{DS}_{K}}\widetilde{A}_{K(n-r)}\right).
$$
Since each $\widetilde{A}_{K(n-r)}$ is irreducible by Lemma~\ref{irreducibletwists} and distinct by Lemma~\ref{uniquetwists}, with $\widetilde{A}_{K(0)}=1$, and since each slope $r\in\mathcal{DS}_{C}$ is distinct, we see that $\mathrm{Red}\left[\widetilde{A}_{K(n-r)}\widetilde{A}_{K(n-s)}\right]=\widetilde{A}_{K(n-r)}\widetilde{A}_{K(n-s)}$ for all $r\neq s$.
Therefore,
$$
A_{D_{n}(C)}=(L-1)\prod_{r\in\mathcal{DS}_{C}}\widetilde{A}_{K(n-r)},
$$
as claimed which completes the proof.
\null\hfill$\square$\\\\

The following corollaries provide many computational examples and are immediate consequences of Theorem~\ref{doublellin}, Remark~\ref{iteratedtorus}, and Corollary~\ref{torconnected}.
The strongly detected boundary slopes of iterated torus knots $[(p_{1},q_{1}),\ldots,(p_{n},q_{n})]$ were noted in~\cite{nz_2017} as $p_{i}q_{i}\prod_{j=1}^{i-1}q_{j}^{2}$ which were also shown to be distinct.
The slopes of $T(p,q)\#T(p',q')$ are also easy to find given a calculation of $A_{T(p,q)\#T(p',q')}$ from Theorem~\ref{gzconnect}.

\begin{corollary}
\label{iteratedtorusdouble}
The $A$-polynomial of the $r$-twisted Whitehead double of an iterated torus knot, denoted $D_{r}[(p_{1},q_{1}),\ldots,(p_{n},q_{n})]$, is given by
\begin{align}
A_{D_{r}[(p_{1},q_{1}),\ldots,(p_{n},q_{n})]}=(L-1)\widetilde{A}_{K(r)}\underset{i=1}{\overset{}{\prod}}\widetilde{A}_{K\left(r-p_{i}q_{i}\Pi_{j=1}^{i-1}q_{j}^{2}\right)}.
\end{align}
\end{corollary}
\begin{corollary}
\label{doubleconnectedsum}
For torus knots $T(p,q),T(p',q')$ with $q\geq q'$, the $A$-polynomial of the $n$-twisted Whitehead double of their connected sum $K=T(p,q)\#T(p',q')$ is given by
\begin{align*}
A_{D_{n}(K)}&\doteq\begin{cases}
(L-1)\widetilde{A}_{K(n)}\widetilde{A}_{K(n-pq)}\widetilde{A}_{K(n-p'q')}\widetilde{A}_{K(n-(pq+p'q'))}&:\hspace{4pt}|p|q\neq|p'|q',\\
(L-1)\widetilde{A}_{K(n)}\widetilde{A}_{K(n-pq)}\widetilde{A}_{K(n-2pq)}&:\hspace{4pt}pq=p'q',\\
(L-1)\widetilde{A}_{K(n)}\widetilde{A}_{K(n-pq)}\widetilde{A}_{K(n+pq)}&:\hspace{4pt}pq=-p'q'.
\end{cases}
\end{align*}
\end{corollary}

%%%%%%%%%%%%%&&&&&&&&&&&%%%%%%%%%%%%%%%%%%%%%%%%%
% Section 8: The r-Twisted WHD of n-Twist Knots %
%%%%%%%%%%%%%&&&&&&&&&&&%%%%%%%%%%%%%%%%%%%%%%%%%
\section{The $r$-Twisted Whitehead Double of $n$-Twist Knots}
\label{whiteheadtwist}
When $C=K(n)$ for $n\neq1,0$, we notice that $\mathrm{Vol}(\mathcal{M}_{C})>0$ since $K(n)$ is hyperbolic; additionally, the $r$-twisted Whitehead double $D_{r}(C)$ will have satellite space $\mathcal{M}_{2}$ in the JSJ-decomposition also with positive hyperbolic volume, $\mathrm{Vol}(\mathcal{M}_{2})>0$.
The $A$-polynomial of $K=D_{r}(C)$ will be more difficult than for the case of graph knots, though it can be found as a factor of the iterated resultant (which typically factors into multiple irreducible polynomial factors):
$$
\mathrm{Res}_{u}\left[\mathrm{Res}_{t}\left[A_{C}(uv^{-r},v),f_{W}\right],\mathrm{Res}_{t}\left[f_{W},F_{W}\right]\right]=
[P_{K(n),r}(L,M)]^{2}[Q_{K(n),r}(L,M)]^{2}.
$$

\begin{remark}
For our computations, we will not heavily distinguish between $D_{r}(K(n))$ and $D_{r}(K(n))^{\ast}$ since it is clear that $D_{r}(C)^{\ast}=D_{-r}(C^{\ast})$ and therefore $D_{r}(K(n))^{\ast}=D_{-r}(J(2,2n)^{\ast})=D_{-r}(J(-2,-2n))$.
However, in this particular case, $K(-1)^{\ast}=K(-1)$ hence we see that $D_{r}(K(-1))^{\ast}=D_{-r}(K(-1))$.
\end{remark}

Since we know that $A_{K(r)}|A_{D_{r}(K(n))}$, we will denote the remaining factor $A_{K(r)}^{-1}\widetilde{A}_{D_{r}(K(n))}$ by $\widetilde{P}_{K(n),r}=\widetilde{F}_{\mathrm{Sat}(K(r),K(n),f)}$ (following the previous notation from Section~\ref{proof1} for $\widetilde{F}_{\mathrm{Sat}(P,C,f)}$), which is computationally equivalent to $P_{K(n),r}$ as stated in Remark~\ref{fig8dbl} using specific calculations; the other factor $Q_{K(n),r}$ is a byproduct of iterated resultant computations.
\label{slopecomputations}
To verify that the factor $Q_{K(n),r}$ is invalid, we use Hoste and Shanahan's table for boundary slope computations~\cite{hs_2007} (here using particular $\mathcal{L}_{3/8}=W$ with $k=1$ in their notation $\mathcal{L}_{\frac{4k-1}{8k}}$), to find boundary slope pairs for $\mathcal{BS}_{W}$:

$$
\begin{tabular}{@{}ccc@{}}
\multicolumn{3}{@{}c}{Table 1. Boundary Slope Pairs for $\mathcal{L}_{3/8}=W$}\\ \hline
\multicolumn{2}{@{}c}{$\partial$-Slopes} & Restrictions \\ \hline
$(0,\varnothing)$ & $(\varnothing,0)$ & \\
$(-4,\varnothing)$ & $(\varnothing,-4)$ & \\
\multicolumn{2}{@{}c}{$(2t^{-1},2t)$} & $0\leq t\leq\infty$ \\
\multicolumn{2}{@{}c}{$(-2t^{-1}-2,-2t)$} & $0\leq t\leq 1$ \\
\multicolumn{2}{@{}c}{$(-2t^{-1},-2-2t)$} & $1\leq t\leq\infty$ \\
\multicolumn{2}{@{}c}{$(-3+s,-3-s)$} & $-1\leq s\leq 1$ \\ \hline
\end{tabular}
$$

For example, we realize that the boundary slopes $0,-4$ will always occur in $\mathcal{BS}_{D_{r}(K(n))}$ by the first two lines of Table 1, and without loss of generality, we use the convention that the attaching boundary slope is in the second component.
We verify that the boundary slope pairs given by~\cite{hs_2007} provide us with the means to compute the boundary slopes of $\mathcal{BS}_{D_{r}(K(n))}$ by the following well-known result:
\begin{lemma}
\label{bdyslopegluing}
Let $C$ be a nontrivial knot, let $L=\ell_{x}\cup\ell_{y}$ with $\ell_{y}$ an unknot, and $f:P\hookrightarrow V$ an embedding with $f(P)=\ell_{x}$ such that $\mathcal{M}_{L}=V-N(f(P))$, let $\phi:\partial\mathcal{M}_{C}\to\partial_{y}\mathcal{M}_{L}$ be the standard gluing map with $\partial_{y}\mathcal{M}_{L}=\partial N(\ell_{y})$,
\begin{align*}
\phi(\mu_{K})&=\lambda_{y}&
\phi(\lambda_{K})&=\mu_{y},
\end{align*}
and so $K=\mathrm{Sat}(P,C,f)$ with $V=\mathcal{M}_{\ell_{y}}$.
Then,
$$
\mathcal{BS}_{K}=\left\{m_{x}\middle|\exists m, m_{y}:m\in\mathcal{BS}_{C},(m_{x},m_{y})\in\mathcal{BS}_{L},\tfrac{1}{m}=m_{y}\right\}\cup\left\{m_{x}\middle|(m_{x},\varnothing)\in\mathcal{BS}_{L}\right\}.
$$
\end{lemma}
\begin{proof}%[Proof of Lemma]
Recall that $m=p/q\in\mathbb{Q}\cup\{\infty\}$ is in $\mathcal{BS}_{K'}$ if there is a properly embedded essential surface $(F,\partial F)\subset(\mathcal{M}_{K'},\partial\mathcal{M}_{K'})$ with $\partial F$ a collection of parallel simple closed curves with slope $p/q$.
Likewise, a slope-pair $(m_{x},m_{y})\in\mathcal{BS}_{L}$ if there is a properly embedded essential surface $F$ in $\mathcal{M}_{L}$ with $\partial_{x} F$ and $\partial_{y}F$ a collection of parallel simple closed curves with slopes $m_{x}$ and $m_{y}$ respectively.
We see immediately that $\left\{m_{x}\middle|(m_{x},\varnothing)\in\mathcal{BS}_{L}\right\}\subset\mathcal{BS}_{K'}$ since any such pair $(m_{x},\varnothing)$ has an associated essential surface $F$ which can also be embedded into $\mathcal{M}_{K'}$.
Likewise, for any slope pair $(m_{x},m_{y})\in\mathcal{BS}_{L}$ with $m_{y}=\tfrac{1}{m}$ for some $m\in\mathcal{BS}_{K}$, we have corresponding essential surfaces $F_{K},F_{L}$, and we may take necessary parallel copies of these surfaces until they agree on the number of boundary components along the gluing torus.
This new surface will be essential in $\mathcal{M}_{K'}$ since it's components are essential in their respective submanifolds, thus $m_{x}\in\mathcal{BS}_{K'}$.

Conversely, if $m_{x}\in\mathcal{BS}_{K'}$, then there exists a properly embedded essential surface $F$ with slope $m_{x}$ along $\partial\mathcal{M}_{K'}$.
If $F$ does not intersect $\mathcal{M}_{K}$ or if $F$ can be isotoped in $\mathcal{M}_{K'}$ so as to not intersect $\mathcal{M}_{K}$, then $(m_{x},\varnothing)\in\mathcal{BS}_{L}$.
However, if $F\cap\mathcal{M}_{K}$ is a nontrivial intersection, then $F\cap\partial\mathcal{M}_{K}$ is a collection of parallel simple closed curves on the torus $\partial\mathcal{M}_{K}$, {\it i.e.} some slope $m$.
This implies that $F=F_{K}\cup_{\phi}F_{L}$ where $F_{K}$ is a properly embedded essential surface with slope $m$ along $\partial\mathcal{M}_{K}$.
The other component $F_{L}$ will exhibit a boundary slope pair $(m_{x},m_{y})\in\mathcal{BS}_{L}$ which must satisfy the gluing relation $\phi$; hence $m_{y}=\tfrac{1}{m}$ and the lemma is proven.
\end{proof}

\begin{remark}
\label{fig8dbl}
We have verified the following formula for $A_{D_{r}(K(-1))}$ for twists $-11\leq r\leq11$:
$$
A_{D_{r}(K(-1))}=(L-1)\widetilde{A}_{K(r)}\widetilde{P}_{K(-1),r},
$$
where the last factor $\widetilde{P}_{K(-1),r}$ in the verified cases is equal to the polynomial $P_{K(-1),r}$ below, computed via resultant methods for $-11\leq r\leq11$,
$$
P_{K(-1),r}=\widetilde{A}_{K(r-4)}\widetilde{A}_{K(r+4)}-L(M^{2}-1)^{3}(M^{2}+1)(L-M^{4})x^{2}y(2x^{2}-y)y^{k(r)}(L+M^{2})^{\varepsilon(r)},
$$
with the polynomial factors $x,y$ as given in Hoste and Shanahan~\cite{hs_2004},
\begin{align*}
x&=(L+M^{2})\widetilde{A}_{K(1)}+\widetilde{A}_{K(-1)},\\
y&=M^{4}(L+M^{2})^{4},
\end{align*}
and the exponents $k(r),\varepsilon(r)$ are given by:
$$
k(r)=\begin{cases}
r-4&:\hspace{4pt}r>4\\
0&:\hspace{4pt}-4< r\leq4\\
-r-4&:\hspace{4pt}r\leq-4,
\end{cases}\hspace{20pt}
\varepsilon(r)=\begin{cases}
-1&:\hspace{4pt}r>4\\
0&:\hspace{4pt}-4< r\leq4\\
1&:\hspace{4pt}r\leq-4.
\end{cases}
$$
\end{remark}

Returning to twist-knot exteriors with attaching map $\phi_{r}$, for a boundary slope $p/q\in\mathcal{BS}_{K(n)}$, this corresponds to an essential surface in $\mathcal{M}_{K(n)}$ whose boundary is in the class $\mu_{K}^{p}\lambda_{K}^{q}\in\pi_{1}(\partial\mathcal{M}_{K(n)})$.
In the $r$-twisted Whitehead double, our boundary slopes will come from $(m_{x},m_{y})\in\mathcal{BS}_{W}$ which correspond to essential surfaces in $\mathcal{M}_{W}$ where the boundary components on $\partial N(\ell_{x})$ have parallel slopes $m_{x}$ and similarly, the boundary components on $\partial N(\ell_{y})$ have parallel slopes $m_{y}$.
We naturally expect to encounter boundary slopes of the form $(m_{x},\varnothing)$ corresponding to a boundary slope that can be isotoped to not intersect with the identified torus $\partial N(\ell_{y})=\partial\mathcal{M}_{K(n)}$.
These boundary slopes justify why $0,-4\in\mathcal{BS}_{D_{r}(K(n))}$ for all values of $n,r\in\mathbb{Z}$.

The more interesting boundary slopes we encounter are derived from boundary slopes $(m_{x},m_{y})\in\mathcal{BS}_{W}$ where $\phi_{r\ast}(m)=m_{y}$ for some $m\in\mathcal{BS}_{K(n)}$ by the above remark.
These boundary slopes will come from the gluing $\phi_{r}$, and so we expect to see:
$$
\phi_{r\ast}(p,q)=[\phi(\mu_{K}^{p}\lambda_{K}^{q})]=[(\lambda_{y})^{p}(\mu_{y}\lambda_{y}^{-r})^{q}]=[\mu_{y}^{q}\lambda_{y}^{p-qr}]=(q,p-qr).
$$
Thus, the boundary slopes in $\mathcal{BS}_{D_{r}(K(n))}$ which correspond to essential surfaces that nontrivially intersect the gluing torus will correspond to boundaries $m_{x}\in\mathbb{Q}\cup\{\infty\}$ where $(m_{x},m_{y})\in\mathcal{BS}_{W}$ and $p/q\in\mathcal{BS}_{K(n)}$ will correspond to $m_{y}=q/(p-qr)$.
This means that we may explicitly compute possible boundary slopes using a modified version of the table from Hoste and Shanahan by seeing when $m_{y}=q/(p-qr)$ for some $p/q\in\mathcal{BS}_{K(n)}$ and which pair $(m_{x},m_{y})$ is present in the table.

Included below are two tables of the computed boundary slopes for all cases of $r$ and $n\leq-1$ using the fact that the boundary slopes of $n$-twist knots are known~\cite{ho_1989}:

\begin{align*}
\mathcal{BS}_{K(n)}=\begin{cases}
\{-4,0,-4n\}&:n\leq-1\\
\{0\}&:n=0\\
\{0,-6\}&:n=1\\
\{-4,0,-4n-2\}&:n\geq2.
\end{cases}
\end{align*}
$$
\begin{tabular}{@{}cccccr@{}} 
\multicolumn{6}{@{}c}{Table 2. Boundary Slope Table for $D_{r}(K(n))$ with $n\leq-1$ via Lemma~\ref{bdyslopegluing}}\\ \hline
$\varnothing$ & $\varnothing$ & $1/(-4-r)$ & $1/(-r)$ & $1/(-4n-r)$ \\ \hline
$-4$ & $0$ & $-4r-16$ & $-4r$ & $-4r-16n$ & $r<-4$ \\
$-4$ & $0$ & $0$ & $32$ & $-16n+16$ & $r=-4$ \\
$-4$ & $0$ & $-4r-18$ & $-4r$ & $-4r-16n$ & $-4<r<0$ \\
$-4$ & $0$ & $-18$ & $0$ & $-16n$ & $r=0$ \\
$-4$ & $0$ & $-4r-18$ & $-4r-2$ & $-4r-16n$ & $0<r<-4n$ \\
$-4$ & $0$ & $16n-18$ & $16n-2$ & $0$ & $r=-4n$ \\
$-4$ & $0$ & $-4r-18$ & $-4r-2$ & $-4r-16n-2$ & $r>-4n$\\ \hline
\end{tabular}
$$
$$
\begin{tabular}{@{}cccccr@{}}
\multicolumn{6}{@{}c}{Table 3. Boundary Slope Table for $D_{r}(K(n))$ with $n\geq2$ via Lemma~\ref{bdyslopegluing}}\\ \hline
$\varnothing$ & $\varnothing$ & $1/(-4-r)$ & $1/(-r)$ & $1/(-4n-r)$ \\ \hline
$-4$ & $0$ & $-4r-16$ & $-4r$ & $-4r-16n-8$ & $r<-4n-2$ \\
$-4$ & $0$ & $16n-8$ & $16n+32$ & $0$ & $r=-4n-2$ \\
$-4$ & $0$ & $-4r-16$ & $-4r$ & $-4r-16n-10$ & $-4n-2<r<-4$ \\
$-4$ & $0$ & $0$ & $16$ & $-16n+6$ & $r=-4$ \\
$-4$ & $0$ & $-4r-18$ & $-4r$ & $-4r-16n-10$ & $-4<r<0$ \\
$-4$ & $0$ & $-18$ & $0$ & $-16n-10$ & $r=0$ \\
$-4$ & $0$ & $-4r-18$ & $-4r-2$ & $-4r-16n-10$ & $r>0$\\ \hline
\end{tabular}
$$

We see that the boundary slope corresponding to $1/(-r)$ is $-4r$ when $r\leq0$, and $-4r-2$ when $r>0$ (regardless of choice of $n$), which are the boundary slopes coming from $\mathcal{BS}_{K(r)}$.
Hence, we see that $\widetilde{P}_{K(-1),r}$ for $r\in\mathbb{Z}$ cannot be equal to $\widetilde{A}_{K(m)}$ for any $m\neq-4$ since the boundary slopes from $K(m)$ are $\{-4,0,-4m\}$ while the strongly detected boundary slopes coming from $P_{K(-1),r}$ are
$$
\begin{cases}
\{-4,0,-4r-16,16-4r\}&:r<-4\\
\{-4,0,32\}&:r=-4\\
\{-4,0,-4r-18,16-4r\}&:-4<r<0\\
\{-4,0,-18,16\}&:r=0\\
\{-4,0,-4r-18,16-4r\}&:0<r<4\\
\{-4,0,-34\}&:r=4\\
\{-4,0,-4r-18,-4r+14\}&:r>4.
\end{cases}
$$
Notice in the special cases of $r=\pm4$, we find that $P_{K(-1),\pm4}$ has slopes identical with $\widetilde{A}_{K(\pm8)}$; however, the polynomials themselves are different by computation, and so $P_{K(-1),r}\neq\widetilde{A}_{K(m)}$ for any $m$.

In practice, we find that exactly one factor $P_{K(-1),r}$ has Newton polygon $\mathrm{Newt}(P_{K(-1),r})$ which exhibits these slopes, while the only other factor observed $Q_{K(-1),r}$ has Newton polygon $\mathrm{Newt}(Q_{K(-1),r})$ which exhibits a slope of $2$ (which is never seen in the predicted slopes for any $r$).
While the computation remains difficult, a formula for the simple case when $n=-1$ is presented (having been verified for $r=-11,\ldots,11$ using the boundary slopes) in Remark~\ref{fig8dbl}.

From the computed examples, it is apparent that the $A$-polynomial of the $r$-twisted Whitehead double of a non-graph knot is not an inherently obvious computation; more optimistically, the $A$-polynomial of $r$-twisted Whitehead doubles still exhibit some connections to the $A$-polynomials of twist knots, seen with the $\widetilde{A}_{K(r+4)}\widetilde{A}_{K(r-4)}$ summand in the expression.

%%%%%%%%%%%%%%%%%%%%%%%%%
% SECTION 9: CONCLUSION %
%%%%%%%%%%%%%%%%%%%%%%%%%
\section{Conclusion}
\label{conclusion}
In summary, we have provided formulas for computing $A$-polynomials of several families of satellite knots; namely, connected sums and iterated cables of pseudo-graph knots and all winding number zero satellites of graph knots.
From this, the $A$-polynomials of all graph knots can be computed once the construction of the graph knot as cables and connected sums is understood, and will have zero logarithmic Mahler measure.
For graph knots, the main property which allows winding number zero satellites to be computed is that their $A$-polynomials have no gaps and they have killing slopes.
Further calculations show that these killing slopes are connected to the knots $f(P)_{r}$ obtained from $(1/q)$-Dehn filling on $\partial V$.

One future goal is a strategy for understanding how to more generally compute the factor $\widetilde{F}_{\mathrm{Sat}(P,C,f)}$ for various families of knots, either recursively or explicitly, broadening the understanding of the $A$-polynomials of satellite knots.
Another direction is to find explicit formulas for the $A$-polynomials of certain knots, thereby extending the applications of the cabling formula and eliminating the need for polynomial reduction in certain cases.
As mentioned, it is unclear whether the $A$-polynomial of a graph knot can also be the $A$-polynomial of a knot with positive hyperbolic volume; more generally, it is unclear whether a satellite knot $K=\mathrm{Sat}(P,C,f)$ with $\mathrm{m}(A_{K})=0$ must also have $\mathrm{Vol}(\mathcal{M}_{K})=0$.
However, Corollary~\ref{graphknotsgz} implies the converse and counterexamples remain difficult to find.

\end{spacing}
\end{document}